\documentclass{amsart}
\usepackage{appendix}
\usepackage{tabularx}
\usepackage{xcolor}
\usepackage[numbers,sort&compress]{natbib}
\usepackage{color}

\newtheorem{theorem}{Theorem}[section]
\newtheorem{lemma}[theorem]{Lemma}

\newtheorem{corollary}[theorem]{Corollary}

\theoremstyle{remark}
\newtheorem{remark}[theorem]{Remark}

\numberwithin{equation}{section}
\allowdisplaybreaks[4]

\begin{document}

\title[Gradient estimates and Parabolic frequency]{Gradient estimates and Parabolic frequency under the Laplacian $G_{2}$ flow}

\author{Chuanhuan Li}
\address{ School of Mathematical Sciences, Laboratory of Mathematics and Complex Systems, Beijing Normal University, Beijing 100875, China}
\curraddr{}
\email{chli@mail.bnu.edu.cn}
\thanks{}

\author{Yi Li$^{\ast}$}
\address{ School of Mathematics and Shing-Tung Yau Center of Southeast University, Southeast University, Nanjing 211189, China\newline
${\quad }$ Shanghai Institute for Mathematics and Interdisciplinary Sciences,
 657 Songhu Road, Yangpu District, Shanghai 200433, China}
\curraddr{}
\email{yilicms@gmail.com, yilicms@seu.edu.cn}
\thanks{*Corresponding author}

\author{Kairui Xu}
\address{ School of Mathematics, Southeast University, Nanjing 211189, China}
\curraddr{}
\email{xukarry@163.com}
\thanks{}



\date{}

\dedicatory{}

\begin{abstract}
    In this paper, we consider the Laplacian $G_{2}$ flow on a closed seven-dimensional manifold $M$ with a closed $G_{2}$-structure. We first obtain the gradient estimates {for} positive solutions of the heat equation under the Laplacian $G_{2}$ flow and then we get the Harnack inequality on spacetime. As an application, we prove the monotonicity of parabolic frequency for positive solutions of the heat equation with bounded Ricci curvature, and get the integral-type Harnack inequality. Besides, we prove the monotonicity of parabolic frequency for solutions of the linear heat equation with bounded Bakry-\'{E}mery Ricci curvature, and then obtain the backward uniqueness. 


\end{abstract}
  
\maketitle

\section{Introduction}

\subsection{Gradient estimates under the Laplacian $G_{2}$ flow}
In \cite{heat equation by Hamilton, L-Y 86}, P. Li, S.-T. Yau and Hamilton obtained the following gradient estimates for positive solutions of the heat equation on a closed Riemannian manifold with Ricci curvature bounded {below}.

${}$

{\bf Theorem A. (Li-Yau, 1986)} Let $(M,g)$ be a closed $n$-dimensional manifold with nonnegative Ricci curvature, and $u=u(x,t)$ be a positive solution of the heat equation on $M\times(0,\infty)$. Then the following estimate 
$$
\frac{|\nabla u|^{2}}{u^{2}}-\frac{\partial_{t}u}{u}\leq\frac{n}{2t}
$$
holds on $M\times(0,\infty)$.

${}$

{\bf Theorem B. (Hamilton, 1993)} Let $(M,g)$ be a closed $n$-dimensional manifold with ${\rm Ric}\geq -Kg$ for some $K\geq0$, and $u=u(x,t)$
be a positive solution of the heat equation with $u(x,t)\leq A$ for all $(x,t)\in M\times(0,\infty)$, where $A$ is a positive constant. Then the following estimate
\begin{equation}
    \frac{|\nabla u|^{2}}{u^{2}}\leq\left(\frac{1}{t}+2K\right)\ln\frac{A}{u}\notag
\end{equation}
holds on $M\times(0,\infty)$.

${}$

These two estimates provide a versatile tool for studying the analytical, topological, and geometrical properties of manifolds.

In 2010, B\v aile\c steanu-Cao-Pulemotov \cite{heat equation under RF by Cao xiaodong} {obtained} the Li-Yau estimate for positive solutions of the heat equation when the metrics $g(t)$ are evolved by the Ricci flow
\begin{align}
\partial_{t}g(t)=-2 \!\ {\rm Ric}(g(t)).
\end{align}
The Ricci flow was introduced by Hamilton in \cite{RF} to study the compact three-manifolds with positive Ricci curvature, which is a special case of the Poincar\'e conjecture finally proved by Perelman in \cite{Perelman 03, Perelman 02}. Hamilton \cite{RF} obtained the short-time existence and uniqueness of the Ricci flow on compact manifolds, and Shi \cite{Shi-89} obtained a short-time solution of the Ricci flow on a complete noncompact manifold and the uniqueness with bounded Riemann curvature was proved by Chen-Zhu in \cite{C-Z 2006}. After that, many people began to study the gradient estimate {for} the positive solutions of the heat equation when the metrics are evolved {by} geometric flows see \cite{heat equation under RHF, Gradient estimate Liu, Sun 2011}.

${}$

In this paper, we first study gradient estimates {for} positive solutions of the heat equation under the Laplacian $G_{2}$ flow for closed $ G_{2}$-structure:
\begin{equation}
  \left \{
       \begin{array}{rl}
          \partial_{ t}\varphi(t)&=\triangle_{\varphi(t)}\varphi(t),\\
           \varphi(0)&=\varphi,
       \end{array}
  \right.
  \label{G2 flow}
\end{equation}
which was introduced by Bryant \cite{G2 structure Bryant} on a smooth $7$-manifold $M$ admitting closed $G_{2}$-structure,
where $\triangle_{\varphi(t)}\varphi(t)=dd^{\ast}_{\varphi(t)}\varphi(t)+d^{\ast}_{\varphi(t)}d\varphi(t)$ is the Hodge Laplacian of $g(t)$ and $\varphi$ is an initial closed $G_{2}$-structure. Here $g(t)$ is the associated Riemannian metric of $\varphi(t)$. Since for a closed $G_{2}$-structure $\varphi$, $\triangle_{\varphi}\varphi=dd^{\ast}_{\varphi}\varphi$, we see that the closedness of $\varphi(t)$ is preserved along the Laplacian $G_{2}$ flow  $\eqref{G2 flow}$. The existence {for} the solution {of} the Laplacian $G_{2}$ flow can be found in \cite{G2 structure Bryant, Fine-Yao 2018, LL G2 flow, local curvature of G2, G2 flow}.

We first consider the following Li-Yau type gradient estimate of the heat equation 
\begin{align}
    \label{heat equation}\partial_{t}u(t)=\Delta_{g(t)}u(t)
\end{align}
under the Laplacian $G_{2}$ flow $\eqref{G2 flow}$, where $\Delta_{g(t)}={\rm tr}_{g(t)}\left(\nabla^{2}_{g(t)}\right)$ is the {trace} Laplacian induced by $g(t)$.

\begin{theorem}\label{theorem 1.1}
    Let $(M,\varphi(t))_{t\in(0,T]}$ be the solution of the Laplacian $G_{2}$ flow $\eqref{G2 flow}$ on a closed $7$-dimensional manifold $M$ with $T<+\infty$ and $-Kg(t)\leq {\rm Ric}(g(t))\leq0$, where $g(t)$ is the Riemannian metric associated with $\varphi(t)$ and $K$ is a positive constant. If $u(t)$ is a positive solution of the heat equation $(\ref{heat equation})$, then on $M\times(0, T]$, the following estimate
    \begin{equation}
        \frac{|\nabla_{g(t)} u(t)|_{g(t)}^{2}}{u^{2}(t)}-\alpha\frac{\partial_{t}u(t)}{u(t)}\leq \frac{7\alpha}{2at}+\left(\frac{49\alpha}{3a}+\frac{105\alpha^{2}-98\alpha}{2a(\alpha-1)}+\frac{7\sqrt{29}\alpha}{2\sqrt{ab}}\right)K
    \end{equation}
holds for any $\alpha>1$ {and} $a,b>0$ with $\displaystyle{a+2b=\frac{1}{\alpha}}$.
\end{theorem}

As an application, we can get the following Harnack inequality on spacetime. 

\begin{corollary}
    Let $(M,\varphi(t))_{t\in(0,T]}$ be the solution of the Laplacian $G_{2}$ flow $\eqref{G2 flow}$ on a closed $7$-dimensional manifold $M$ with $T<+\infty$ and $-Kg(t)\leq {\rm Ric}(g(t))\leq0$, where $g(t)$ is the Riemannian metric associated with $\varphi(t)$ and $K$ is a positive constant. If $u(t)$ is a positive solution of the heat equation $(\ref{heat equation})$, then for $(x,t_{1}), (y, t_{2})\in M\times(0, T]$ with $t_{1}<t_{2}$, we have 
    $$u(x,t_{1})\leq u(y,t_{2})\left(\frac{t_{2}}{t_{1}}\right)^{\frac{7}{2a}}\exp\left\{\int_{0}^{1} \left[\frac{\alpha|\gamma^{\prime}(s)|_{\sigma(s)}^{2}}{4(t_{2}-t_{1})}+(t_{2}-t_{1})C_{a,b,\alpha}K\right] ds\right\},$$
    where $\alpha>1$, 
    $$C_{a,b,\alpha}=\frac{49}{3a}+\frac{105\alpha-98}{2a(\alpha-1)}+\frac{7\sqrt{29}}{2\sqrt{ab}},
    $$
    $a,b>0$ with $\displaystyle{a+2b=\frac{1}{\alpha}}$, $\gamma(s)$ is a geodesic curve connecting $x$ and $y$ with $\gamma(0)=y$ and $\gamma(1)=x$, and $|\gamma^{\prime}(s)|_{\sigma(s)}$ is the length of the vector $\gamma^{\prime}(s)$ at $\sigma(s)=(1-s)t_{2}+st_{1}$.
\end{corollary}

For the Hamilton type gradient estimate of the heat equation under the Laplacian $G_{2}$ flow $\eqref{G2 flow}$, we have

\begin{theorem}\label{theorem 1.3}
    Let $(M,\varphi(t))_{t\in(0,T]}$ be the solution of the Laplacian $G_{2}$ flow $\eqref{G2 flow}$ on a closed $7$-dimensional manifold $M$ with $T<+\infty$ and $-Kg(t)\leq {\rm Ric}(g(t))\leq0$, where $g(t)$ is the Riemannian metric associated with $\varphi(t)$ and $K$ is a positive constant. If $u(t)$ is a positive solution of the heat equation, then on $M\times(0, T]$, the following estimate
    \begin{equation}
         |\nabla_{g(t)} u(t)|_{g(t)}^{2}\leq\frac{u(t)}{t}\left[u(t)\ln\frac{A}{u(t)}+\lambda A^{2}-\lambda\eta^{2}\right]
    \end{equation}
holds, where $\displaystyle{\eta=\min_{M}u(0)}$, $\displaystyle{A=\max_{M}u(0)}$ and $\lambda$ is a constant depending on $K$, $\eta$ and $T$.
\end{theorem}
 
\subsection{Parabolic frequency under the Laplacian $G_{2}$ flow}

In 1979, the (elliptic) frequency functional for a harmonic function $u(x)$ on $\mathbb{R}^{n}$ was introduced by Almgren in \cite{Dirichlet problem by Almgren}, which was defined by
$$
N(r)=\frac{\displaystyle{r\int_{B(r,p)}|\nabla u(x)|^{2}dx}}{\displaystyle{\int_{\partial B(r,p)}u^{2}(x)d\sigma}},
$$
where $d\sigma$ is the induced $(n-1)$-dimensional Hausdorff measure on $\partial B(r,p)$, $B(r,p)$ is the ball in $\mathbb{R}^{n}$ and $p$ is a fixed point in $\mathbb{R}^{n}$. Almgren obtained that $N(r)$ is monotone nondecreasing for $r$, and he used this property to investigate the local regularity of harmonic functions and minimal surfaces. Next, Garofalo and Lin \cite{Mon prop by Garofalo, Unique continuation by Garofalo} considered the monotonicity of frequency functional on Riemannian manifolds to study the unique continuation for elliptic operators. The frequency functional was also used to estimate the size of nodal sets in \cite{LA-2018, LA-2018 Yau}. For more applications, see \cite{Harmonic functions by Colding, H-H-L 1998, H-L 1994, frequency by Lin, Zelditch}.

The parabolic frequency for the solution of the heat equation on $\mathbb{R}^{n}$ was introduced by Poon in \cite{Unique continuation by Poon}, and  Ni \cite{N-2015} considered the case when $u(t)$ is a holomorphic function, both of them showed that the parabolic frequency is nondecreasing. Besides, on Riemannian manifolds, the monotonicity of the parabolic frequency was obtained by Colding and Minicozzi \cite{frequency on manifold} through the drift Laplacian operator. Using the matrix Harnack's inequality in \cite{heat equation by Hamilton}, Li and Wang \cite{Parabolic frequency by Li-Wang} investigated the parabolic frequency on compact Riemannian manifolds and the 2-dimensional Ricci flow.

In \cite{frequency on RF}, Baldauf-Kim  defined the following parabolic frequency for a solution $u(t)$ of the heat equation 
$$U(t)=-\frac{\tau(t)\|\nabla_{g(t)}u(t)\|_{L^{2}(d\nu)}^{2}}{\|u(t)\|_{L^{2}(d\nu)}^{2}}\cdot\exp\left\{\displaystyle{-\int_{t_{0}}^{t}\frac{1-\kappa(s)}{\tau(s)}ds}\right\},
$$
where $t\in[t_{0},t_{1}]\subset(0,T)$, $\tau(t)$ is the backwards time, $\kappa(t)$ is the time-dependent function and $d\nu$ is the weighted measure. They proved that parabolic frequency $U(t)$ for the solution of the heat equation is monotone increasing along the Ricci
flow with the bounded Bakry-\'{E}mery Ricci curvature and obtained the backward uniqueness. Baldauf, Ho and Lee derived analogous result to the mean curvature flow in \cite{frequency on MCF}.

Recently, Liu and Xu studied the monotonicity of parabolic frequency for the weighted $p$-Laplacian heat equation on Riemannian manifolds in \cite{Liu-xu 2022}, and they obtained a theorem of Hardy–\' Polya–Szeg\" o on K\" ahler manifolds under the K\" ahler-Ricci flow.  In \cite{Li-Zhang 2023}, Li and Zhang derived the matrix Li-Yau-Hamilton estimates for positive solutions to the heat equation and the backward conjugate heat equation under the Ricci flow, and then applied these estimates to study the monotonicity of the parabolic frequency.

In \cite{LLX-2023}, the authors studied the monotonicity of parabolic frequency under Ricci flow and the Ricci-harmonic flow on manifolds. They considered two cases: one is the monotonicity of parabolic frequency for the solution of the linear heat equation with bounded Bakry-\' Emery Ricci curvature, and the other case is the monotonicity of parabolic frequency for the solution of the heat equation with bounded Ricci curvature.

Inspired by \cite{LLX-2023}, we first study the parabolic frequency for the solution of the heat equation $\eqref{heat equation}$ under the Laplacian $G_{2}$ flow $\eqref{G2 flow}$ with bounded Ricci curvature.
The parabolic frequency for the positive solution of the heat equation $\eqref{heat equation}$ is defined by
\begin{align}
    U(t)=\exp&\left\{-\int_{t_{0}}^{t}\left[\frac{h'(s)}{h(s)}-\frac{2}{3}R_{0}+\frac{28+C_{1}(A,\eta)}{s}+cK+\frac{7}{2}C(s)\right]ds\right\}\notag\\
&\cdot\frac{\displaystyle{h(t)\int_{M}|\nabla_{g(t)}u(t)|_{g(t)}^{2}d \mu_{g(t)}}}{\displaystyle{\int_{M}u^{2}(t)\!\ d\mu_{g(t)}}}\notag
\end{align}
where $h(t)$ is a time-dependent function, $K$ and $c$ are both positive constants, 
$$
R_{0}=\min_{M\times[t_{0},t_{1}]}R(t), \ \ \ C_{1}(A,\eta)=\ln\frac{A}{\eta}+\lambda\frac{A^{2}}{\eta}, \ \ C(t)=\frac{C_{1}(A,\eta)}{t}
$$
and $\lambda$ is the constant in Theorem \ref{theorem 1.3},
$$
\eta=\min_{M}u(0), \ \ \ A=\max_{M}u(0). 
$$
Observe that, $A$ and $\eta$ are both positive constants. Using Theorem \ref{theorem 1.1} and Theorem \ref{theorem 1.3} as an application, we have

\begin{theorem}
Let $(M,\varphi(t))_{t\in[0,T]}$ be the solution of the Laplacian $G_{2}$ flow $\eqref{G2 flow}$ on a closed $7$-dimensional manifold $M$  with $T<+\infty$ and $-Kg(t)\leq {\rm Ric}(g(t))\leq0$, where $g(t)$ is the Riemannian metric associated with $\varphi(t)$ and $K$ is a positive constant. If $u(t)$ is a positive solution of the heat equation $\eqref{heat equation}$ with $\eta\leq u(0)\leq A$, then the following holds.

\begin{itemize} 

\item[(i)] If $h(t)$ is a negative time-dependent function, then the parabolic frequency $U(t)$ is monotone increasing along the Laplacian $G_{2}$ flow.

\item[(ii)] If $h(t)$ is a positive time-dependent function, then the parabolic frequency $U(t)$ is monotone decreasing along the Laplacian $G_{2}$ flow.

\end{itemize}
\end{theorem}

Besides, we consider the parabolic frequency for the solution of the linear heat equation 
\begin{equation}
    (\partial_{t}-\Delta_{g(t)})u(t)=a(t)\!\ u(t)\label{lhe}
\end{equation}
under the Laplacian $G_{2}$ flow $\eqref{G2 flow}$ with bounded Bakry-\'{E}mery Ricci curvature, where $a(t)$ is a time-dependent smooth function. The parabolic frequency is defined by  
\begin{align}
U(t)=\exp\left\{-\int_{t_{0}}^{t}\left[-\frac{2}{3}R_{0}+\frac{h'(s)+\kappa(s)}{h(s)}\right] ds\right\}\frac{\displaystyle{h(t)\int_{M}|\nabla_{g(t)}u(t)|_{g(t)}^{2}d \mu_{g(t)}}}{\displaystyle{\int_{M}u^{2}(t)\!\ d\mu_{g(t)}}},\notag
\end{align}
where $\displaystyle{R_{0}=\min_{M\times[t_{0},t_{1}]}R(t)}$, $h(t)$ and $\kappa(t)$ are both time-dependent smooth functions. Then we get the following theorem, where ${\rm Ric}_{f(t)}$ is given in $\eqref{Baker}$.

\begin{theorem}\label{theorem 1.5}
Let $(M,\varphi(t))_{t\in[0,T]}$ be the solution of the Laplacian $G_{2}$ flow $\eqref{G2 flow}$ on a closed $7$-dimensional manifold $M$  with $T<+\infty$ and $\displaystyle{\text{\rm Ric}_{f(t)}\leq \frac{\kappa(t)}{2h(t)}g(t)}$, where $g(t)$ is the Riemannian metric associated with $\varphi(t)$. Then the following holds.

\begin{itemize} 

\item[(i)] If $h(t)$ is a negative time-dependent function, then the parabolic frequency $U(t)$ is monotone increasing along the Laplacian $G_{2}$ flow.

\item[(ii)] If $h(t)$ is a positive time-dependent function, then the parabolic frequency $U(t)$ is monotone decreasing along the Laplacian $G_{2}$ flow.

\end{itemize}
\end{theorem}

The backward uniqueness of solutions to parabolic equations has been the object of consistent study for at least half a century. There are already many results for heat operators concerning it in various domains, such as the exterior domain \cite{BU in exterior domain}, the half-space \cite{BU in half-space} and some cones \cite{BU in cones 2012, BU in cones}. For the heat equation on manifolds, Colding and Minicozzi \cite{frequency on manifold} obtained the backward uniqueness result. Kotschwar showed a backward uniqueness result to Ricci flow in \cite{BU on RF} and gave a general backward uniqueness theorem in \cite{BU on GE}. For more backward uniqueness results of geometric flows, see \cite{BU on YF, BU on open MCF, BU on MCF 2019, BU on inverse MCF, BU on MCF by zhang}.

As an application of Theorem \ref{theorem 1.5}, we get the following backward uniqueness.
\begin{corollary}
Let $(M,\varphi(t))_{t\in[0,T]}$ be the solution of the Laplacian $G_{2}$ flow $\eqref{G2 flow}$ on a closed $7$-dimensional manifold $M$  with $T<+\infty$ and $\displaystyle{\text{\rm Ric}_{f(t)}\leq \frac{\kappa(t)}{2h(t)}g(t)}$, where $g(t)$ is the Riemannian metric associated with $\varphi(t)$ .  If $u(t_{1})=0$, then $u(t)\equiv 0$ for any $t\in[t_{0},t_{1}]\subset(0,T)$.
\end{corollary}

For the general parabolic equation, we have

\begin{theorem}
Let $(M,\varphi(t))_{t\in[0,T]}$ be the solution of the Laplacian $G_{2}$ flow $\eqref{G2 flow}$ on a closed $7$-dimensional manifold $M$  with $T<+\infty$ and $\displaystyle{\text{\rm Ric}_{f(t)}\leq \frac{\kappa(t)}{2h(t)}g(t)}$, where $g(t)$ is the Riemannian metric associated with $\varphi(t)$ and $h(t)$ is a negative time-dependent function.
    Suppose $u(t):M\times[t_{0},t_{1}]\rightarrow \mathbb{R}$ satisfies 
    $$
    \left|\left(\partial_{t}-\Delta_{g(t)}\right) u(t)\right|\leq C(t)\left[|\nabla_{g(t)} u(t)|_{g(t)}+|u(t)|\right]
    $$ 
    along the Laplacian $G_{2}$ flow $\eqref{G2 flow}$. If $u(t_{1})=0$, then $u(t)\equiv 0$ for all $t\in[t_{0},t_{1}]\subset(0,T)$.
\end{theorem}

${}$

We give an outline of this paper. We review the basic theory in Section \ref{section2} about $G_{2}$-structures, $G_{2}$-decompositions of $2$-forms and $3$-forms, and the torsion tensors of $G_{2}$-structures. We also calculate the conjugate heat equation under the Laplacian $G_{2}$ flow $\eqref{G2 flow}$. Section \ref{section 3} proves the Li-Yau type gradient estimate and Hamilton type gradient estimate under the Laplacian $G_{2}$ flow $\eqref{G2 flow}$ with bounded Ricci curvature, and as an application, we get the Harnack inequality on spacetime. In Section \ref{section 4}, using the Li-Yau type gradient estimate and Hamilton type gradient estimate, we prove the monotonicity of parabolic frequency for the solution of the linear equation $\eqref{lhe}$ under the Laplacian $G_{2}$ flow $\eqref{G2 flow}$ with bounded Ricci curvature, then we get the integral-type Harnack inequality. In Section \ref{section 5}, we consider the monotonicity of parabolic frequency for the heat equation and the general parabolic equation under the Laplacian $G_{2}$ flow $\eqref{G2 flow}$ with bounded  Bakry-\'{E}mery Ricci curvature, and obtain the backward uniqueness.

\section{$G_{2}$-structure, Notations and definitions }\label{section2}

In this section, we introduce the $G_{2}$-structure on manifolds, $G_{2}$-decompositions, the torsion tensor, some notations, and definitions.

\subsection{$G_{2}$-structure on smooth manifolds}
Let $\mathbb{O}$ be the octonions (exceptional division algebra), from the vector cross product ``$\times$" on ${\rm Im}\ \mathbb{O}$, we can define the 3-form by
$$\phi(a,b,c):=\frac{1}{2}\langle a,[b,c]\rangle=\langle a\times b,c\rangle\quad\quad\quad\text{for}\ a,b,c\in {\rm Im}\ \mathbb{O}.$$
Let $\{e_{1},e_{2},\cdots,e_{7}\}$ denote the standard basis of $\mathbb{R}^{7}$ and $\{e^{1},e^{2},\cdots,e^{7}\}$ be its dual basis. Using the octonion multiplication table, one can show that
$$\phi=e^{123}+e^{145}+e^{167}+e^{246}-e^{257}-e^{347}-e^{356},$$
where $e^{ijk}:=e^{i}\wedge e^{j}\wedge e^{k}$. When we fix $\phi$, the subgroup of ${\rm GL}(7,\mathbb{R})$ is the exceptional Lie group $G_{2}$, which is a compact, connected, simple $14$-dimensional Lie subgroup of ${\rm SO}(7)$. In fact, $G_{2}$ acts irreducibly on $\mathbb{R}^{7}$ and preserves the metric and orientation for which $\{e_{1},e_{2},\cdots,e_{7}\}$ is an oriented orthonormal basis. Note that $G_{2}$ also preserves the $4$-form
$$\ast_{\phi}\phi=e^{4567}+e^{2367}+e^{2345}+e^{1357}-e^{1346}-e^{1256}-e^{1247},$$
where $\ast_{\phi}$ is the Hodge star operator determined by the metric and orientation.

\begin{remark}
    The vector cross product ``$\times$" is an algebraic structure defined in a normed division algebra. Therefore, the $G_{2}$-structure can only be defined in the 7-dimensional manifold. For more details, see \cite{Spiros introduce to G2}.
\end{remark}

For a smooth $7$-manifold $M$ and a point $x\in M$, we define as in \cite{local curvature of G2,G2 flow}
\begin{align}
    \wedge_{+}^{3}(T_{x}^{\ast}M):=\left\{\varphi_{x}\in\wedge^{3}(T_{x}^{\ast}M)\ \Big{|}\ h^{\ast}\phi=\varphi_{x},\
    \text{for\ invertible}\ h\in\text{Hom}_{\mathbb{R}}(T_{x}^{\ast}M,\mathbb{R}^{7})\right\}\notag
\end{align}
and the bundle
\begin{align}
    \wedge_{+}^{3}(T^{\ast}M):=\bigcup_{x\in M}\wedge_{+}^{3}(T_{x}^{\ast}M)\notag.
\end{align}
We call a section $\varphi$ of $\wedge_{+}^{3}(T^{\ast}M)$ a {\it positive $3$-form} on $M$ or a {\it $G_{2}$-structure} on $M$, and denote the space of positive 3-form by $\Omega^{3}_{+}(M)$. The existence of $G_{2}$-structure is equivalent to the
property that $M$ is oriented and spin, which is equivalent to the vanishing of the first two Stiefel-Whitney classes $\omega_{1}(TM)$ and $\omega_{2}(TM)$. For more details, see Theorem 10.6 in \cite{Lawson-Michelsohn}.

For a $3$-form $\varphi$, we define a $\Omega^{7}(M)$-valued bilinear form $\text{B}_{\varphi}$ by
$$\text{B}_{\varphi}(u,v)=\frac{1}{6}(u\lrcorner\varphi)\wedge(v\lrcorner\varphi)\wedge\varphi,$$
where $u, v$ are tangent vectors on $M$ and $``\lrcorner"$ is the interior multiplication operator (\textit{Here we use the orientation in \cite{G2 structure Bryant}}). Then we can see that any $\varphi\in\Omega^{3}_{+}(M)$ determines a Riemannian metric $g_{\varphi}$ and an orientation $d V_{\varphi}$, hence the Hodge star operator $\ast_{\varphi}$ and the associated $4$-form
$$\psi:=\ast_{\varphi}\varphi$$
can also be uniquely determined by $\varphi$. 

The group $G_{2}$ acts irreducibly on $\mathbb{R}^{7}$ (and hence on $\wedge^{1}(\mathbb{R}^{7})^{\ast}$ and $\wedge^{6}(\mathbb{R}^{7})^{\ast}$), but it
acts reducibly on $\wedge^{k}(\mathbb{R}^{7})^{\ast}$ for $2\leq k\leq 5$. Hence a $G_{2}$ structure $\varphi$ induces splittings
of the bundles $\wedge^{k}(T^{\ast}M)(2\leq k\leq5)$ into direct summands, which we denote by
$\wedge^{k}_{l}(T^{\ast}M,\varphi)$ with $l$ being the rank of the bundle. We let the space of sections
of $\wedge^{k}_{l}(T^{\ast}M,\varphi)$ be $\Omega^{k}_{l}(M)$. Define the natural projections
$$\pi^{k}_{l}:\Omega^{k}(M)\longrightarrow \Omega^{k}_{l}(M),\ \ \alpha\longmapsto \pi^{k}_{l}(\alpha).$$
Then we have
\begin{align}
    \Omega^{2}(M)&=\Omega^{2}_{7}(M)\oplus\Omega^{2}_{14}(M),\notag\\
    \Omega^{3}(M)&=\Omega^{3}_{1}(M)\oplus\Omega^{3}_{7}(M)\oplus\Omega^{3}_{27}(M)\notag.
\end{align}
where each component is determined by
\begin{align}
    \Omega^{2}_{7}(M)&=\{X\lrcorner\varphi:X\in C^{\infty}(TM)\}=\{\beta\in\Omega^{2}(M):\ast_{\varphi}(\varphi\wedge\beta)=2\beta\},\notag\\
    \Omega^{2}_{14}(M)&=\{\beta\in\Omega^{2}(M):\psi\wedge\beta=0\}=\{\beta\in\Omega^{2}(M):\ast_{\varphi}(\varphi\wedge\beta)=-\beta\},\notag
\end{align}
and
\begin{align}
    \Omega^{3}_{1}(M)&=\{f\varphi:f\in C^{\infty}(M)\},\notag\\
    \Omega^{3}_{7}(M)&=\{\ast_{\varphi}(\varphi\wedge\alpha):\alpha\in\Omega^{1}(M)\}=\{X\lrcorner\psi:X\in C^{\infty}(TM)\},\notag\\
    \Omega^{3}_{27}(M)&=\{\eta\in\Omega^{3}(M):\eta\wedge\varphi=\eta\wedge\psi=0\}.\notag
\end{align}

\begin{remark}\label{remark2.1}
    $\Omega^{4}$ and $\Omega^{5}$ have the corresponding decompositions by Hodge duality. The more details for $G_{2}$-decompositions see \cite{G2 structure Bryant, Spiros introduce to G2}.
\end{remark}

By the definition  of $G_{2}$ decompositions, we can find unique differential forms
$\tau_{0}\in\Omega^{0}(M),\tau_{1},\widetilde{\tau}_{1}\in\Omega^{1}(M),\tau_{2}\in\Omega^{2}_{14}(M)$ and $\tau_{3}\in\Omega^{3}_{27}(M)$ such that (see \cite{G2 structure Bryant})
\begin{align}
    d\varphi&=\tau_{0}\psi+3\!\ \tau_{1}\wedge\varphi+\ast_{\varphi}\tau_{3},\\
    d\psi&=4\!\ \widetilde{\tau}_{1}\wedge\psi+\tau_{2}\wedge\varphi.
\end{align}
In fact, Karigiannis\cite{Spiros flow of G2} proved that $\tau_{1}=\widetilde{\tau}_{1}$. We call $\tau_{0}$ the \textit{scalar torsion}, $\tau_{1}$ the \textit{vector torsion}, $\tau_{2}$ the \textit{Lie algebra torsion}, and $\tau_{3}$ the \textit{symmetric traceless torsion}. We also call $\tau_{\varphi}:=\{\tau_{0},\tau_{1},\tau_{2},\tau_{3}\}$ the intrinsic torsion forms of the $G_{2}$-structure $\varphi$.

If $\varphi$ is closed, which means $d\varphi=0$, then $\tau_{0},\tau_{1},\tau_{3}$ are all zero, so the only nonzero torsion form is 
$$\tau\equiv\tau_{2}=\frac{1}{2}(\tau_{2})_{ij}dx^{i}\otimes dx^{j}=\frac{1}{2}\tau_{ij}dx^{i}\otimes dx^{j}.$$
 Then from \cite{Spiros flow of G2,Spiros introduce to G2}, the full torsion tensor $\mathbf{T}=\mathbf{T}_{ij}dx^{i}\otimes dx^{j}$ satisfies the followings
$$
\mathbf{T}_{ij}=-\mathbf{T}_{ji}=-\frac{1}{2}(\tau_{2})_{ij}\ \ \text{or equivalently}\ \ \mathbf{T}=-\frac{1}{2}\tau,
$$
so that $\mathbf{T}$ is a skew-symmetric 2-tensor or a $2$-form.

\subsection{ The Laplacian $G_{2}$ flow and some notations }
In this subsection, we introduce the Laplacian $G_{2}$ flow, some notations, and definitions which will be used in the sequel. We use the notations in Hamilton’s paper \cite{RF}, $\nabla_{g}$ is the Levi-Civita connection induced by $g$, $\text{\rm Ric}(g)$, $R_{g}$, $dV_{g}$ are Ricci curvature, scalar curvature, and volume form, respectively. The Laplacian of the smooth time-dependent function $f(t)$ with respect to a family of Riemannian metrics $g(t)$ is
$$
\Delta_{g(t)}f(t)=g^{ij}(t)\left[\partial_{i}\partial_{j}f(t)-\Gamma_{ij}^{k}(t)\partial_{k}f(t)\right],
$$
where $\Gamma_{ij}^{k}(t)$ is the Christoffel symbol of $g(t)$ and $\displaystyle{\partial_{i}=\frac{\partial}{\partial x^{i}}}$.

In \cite{G2 structure Bryant}, Bryant introduced the following Laplacian $G_{2}$ flow  on a smooth $7$-manifold $M$ admitting closed $G_{2}$-structures
\begin{equation}
  \left \{
       \begin{array}{rl}
          \partial_{ t}\varphi(t)&=\triangle_{\varphi(t)}\varphi(t),\\
           \varphi(0)&=\varphi,
       \end{array}
  \right.\notag
\end{equation}
under the Laplacian $G_{2}$ flow. From \cite{G2 flow}, we see that the associated metric tensor $g(t)$ evolves by 
\begin{align}
\partial_{t}g(t)=-2\!\ {\rm Sic}(g(t)),
\end{align}
where $${\rm Sic}(g(t))={\rm Ric}(g(t))+\frac{1}{3}|{\bf T}(t)|_{g(t)}^{2}g(t)+2\widehat{\bf T}(t)$$
is the symmetric $2$-tensor and its components are given by
\begin{equation}
S_{ij}=R_{ij}+\frac{1}{3}|{\bf T}(t)|^{2}_{g(t)}g_{ij}+2\widehat{{\bf T}}_{ij},
\end{equation}
and $\widehat{\mathbf{T}}_{ij}=\mathbf{T}_{i}^{\ k}\mathbf{T}_{kj}$. In \cite{local curvature of G2,G2 flow}, we see that $R_{g(t)}=-|{\bf T}(t)|^{2}_{g(t)}$ and ${\bf T}_{ij}$ is skew-symmetric, then we have that
\begin{align}
    {\rm tr}_{g(t)}\Big({\rm Sic}(g(t))\Big)&=R_{g(t)}+\frac{7}{3}|{\bf T}(t)|_{g(t)}^{2}-2|{\bf T}(t)|_{g(t)}^{2}\notag\\
    &=-|{\bf T}(t)|_{g(t)}^{2}+\frac{7}{3}|{\bf T}(t)|_{g(t)}^{2}-2|{\bf T}(t)|_{g(t)}^{2}=\frac{2}{3}R_{g(t)}\notag.
\end{align}
Under the Laplacian $G_{2}$ flow $\eqref{G2 flow}$, for any smooth functions $u(t),v(t)$ with
$$
\int_{M}u(T)v(T)dV_{g(T)}=\int_{M}u(0)v(0)dV_{g(0)},
$$ we have that 
\begin{align}
    &\quad\int_{0}^{T}\int_{M}v(t)\left(\partial_{t}-\Delta_{g(t)}\right)u(t)\!\ dV_{g(t)}dt\notag\\
    &=\int_{0}^{T}\int_{M}\left[-u(t)\partial_{t}v(t)+\frac{2}{3}v(t)u(t)R_{g(t)}-u(t)\Delta_{g(t)} v(t)\right]dV_{g(t)}dt\notag\\
    &=\int_{0}^{T}\int_{M}u(t)\left(\frac{2}{3}R_{g(t)}-\partial_{t}-\Delta_{g(t)}\right)v(t)\!\ dV_{g(t)}dt\notag.
\end{align}
Let $\tau(t)=T-t$ be the backward time. For any time-dependent smooth function $f(t)$ on $M$, we denote 
$${\bf K}(t)=(4\pi\tau(t))^{-\frac{7}{2}}e^{-f(t)}$$
to be the positive solution of the conjugate heat equation
\begin{equation}
\partial_{t}{\bf K}(t)=-\Delta_{g(t)} {\bf K}(t)+\frac{2}{3}R_{g(t)}{\bf K}(t).
\end{equation}
From the definition of ${\bf K}(t)$, we can calculate the smooth function $f(t)$ satisfies the following equation
\begin{equation}
\partial_{t}f(t)=-\Delta_{g(t)} f(t){{-}}\frac{2}{3}R_{g(t)}+|\nabla_{g(t)} f(t)|_{g(t)}^{2}+\frac{7}{2\tau(t)} .
\end{equation}
We can define the weighted measure 
\begin{equation}
d\mu_{g(t)}:= {\bf K}(t) dV_{g(t)}=(4\pi\tau(t))^{-\frac{7}{2}}e^{-f(t)}dV_{g(t)},\ \ \int_{M} d\mu_{g(t)}=1.
\end{equation}
And the volume form $dV_{g(t)}$ satisfies
$$\partial_{t}(dV_{g(t)})=-\frac{2}{3}R_{g(t)}dV_{g(t)},$$
thus, the conjugate heat kernel measure $d\mu_{g(t)}$ is evolved by
\begin{equation}
\partial_{t}(d\mu_{g(t)})=-(\Delta_{g(t)} {\bf K}(t))dV_{g(t)}=-\frac{\Delta_{g(t)} {\bf K}(t)}{{\bf K}(t)}d\mu_{g(t)}.
\label{volume}
\end{equation}

On the weighted Riemannian manifold $(M^{n},g(t),d\mu_{g(t)})$, the weighted Bochner formula for any smooth function  $u$ is as follow 
\begin{align}
\label{Bochner}
\Delta_{g(t),f(t)}\left(|\nabla_{g(t)} u|_{g(t)}^{2}\right)=&2\left|\nabla_{g(t)}^{2} u\right|_{g(t)}^{2}+2\left\langle\nabla_{g(t)} u,\nabla_{g(t)}\triangle_{g(t),f(t)}u\right\rangle_{g(t)}\\
&+2\text{\rm Ric}_{f(t)}\left(\nabla_{g(t)} u,\nabla_{g(t)} u\right)\notag,
\end{align}
where
\begin{align}
\text{\rm Ric}_{f(t)}:=\text{\rm Ric}(g(t))+\nabla_{g(t)}^{2} f(t)
\label{Baker}
\end{align}
is the Bakry-\'{E}mery Ricci tensor introduced in \cite{B-E Ricci}, and
\begin{equation}
\Delta_{g(t),f(t)}u:=e^{f(t)}\text{div}_{g(t)}\left(e^{-f(t)}\nabla_{g(t)}{u}\right)=\Delta_{g(t)} u-\left\langle\nabla_{g(t)} f(t),\nabla_{g(t)} u\right\rangle_{g(t)}
\end{equation}
is the drift Laplacian operator for any smooth function $u$.

\section{Gradient estimates under Laplacian $G_{2}$ flow }\label{section 3}

In this section, we consider the Li-Yau type gradient estimate and Hamilton type gradient estimate of the heat equation$\eqref{heat equation}$ under the Laplacian $G_{2}$ flow $\eqref{G2 flow}$. Since $R_{g(t)}=-|{\mathbf{T}}(t)|^{2}_{g(t)}$, which implies the scalar curvature is non-positive here, some methods of gradient estimate require the non-negative Ricci curvature condition, which can't hold in this circumstance. Inspired by Liu in \cite{Gradient estimate Liu}, we weaken the curvature constraints and obtain the gradient estimate for the solution of the heat equation when the metric is evolved by the Laplacian $G_{2}$ flow $(\ref{G2 flow})$.

\begin{theorem}\label{theorem 3.1}
    Let $(M,\varphi(t))_{t\in(0,T]}$ be the solution of the Laplacian $G_{2}$ flow $\eqref{G2 flow}$ on a closed $7$-dimensional manifold $M$ with with $T<+\infty$ and $-Kg(t)\leq {\rm Ric}(g(t))\leq0$, where $g(t)$ is the Riemannian metric associated with $\varphi(t)$ and $K$ is a positive constant. If $u(t)$ is a positive solution of the heat equation, then on $M\times(0, T]$ the following estimate
    \begin{equation}\label{3.1}
        \frac{|\nabla_{g(t)} u(t)|^{2}_{g(t)}}{u(t)^{2}}-\alpha\frac{\partial_{t}u(t)}{u(t)}\leq \frac{7\alpha}{2at}+\left(\frac{49\alpha}{3a}+\frac{105\alpha^{2}-98\alpha}{2a(\alpha-1)}+\frac{7\sqrt{29}\alpha}{2\sqrt{ab}}\right)K
    \end{equation}
holds for any $\alpha>1$ and $a,b>0$ with $\displaystyle{a+2b=\frac{1}{\alpha}}$.
\end{theorem}

\begin{proof}
    We first set $f=\ln u$ and
    \begin{equation}
        F=t\left(|\nabla f|^{2}-\alpha \partial_{t}f\right)\notag.
    \end{equation}
    Observe that, $\eqref{3.1}$ is true when $F< 0$, hence we always assume that $F\geq 0$ in the sequel.
Some computations show that
    \begin{align}
       \Delta\left(|\nabla f|^{2}\right)&=\sum_{1\leq i,j\leq 7}\left(2f_{ij}^{2}+2f_{i}f_{jji}+2R_{ij}f_{i}f_{j}\right)\notag,\\
        \partial_{t}(\Delta f)&=\partial_{t}\left(g^{ij}\nabla_{i}\nabla_{j}f\right)=\sum_{1\leq i,j\leq 7}\left(2R_{ij}f_{ij}+4\widehat{\bf T}_{ij}f_{ij}\right)+\frac{2}{3}|\mathbf{T}|^{2}\Delta f+\Delta(\partial_{t}f)\notag.
    \end{align}
Combining these two equations we have that
    \begin{align}
        \Delta F&=t\Big(\Delta\left(|\nabla f|^{2}\right)-\alpha\Delta(\partial_{t}f)\Big)\notag\\     &=t\Bigg(2\sum_{1\leq i,j\leq 7}f_{ij}^{2}+2\sum_{1\leq i,j\leq 7}f_{i}f_{jji}+2\sum_{1\leq i,j\leq 7}R_{ij}f_{i}f_{j}-\alpha\partial_{t}(\Delta f)\notag\\
        &\quad+2\alpha\sum_{1\leq i,j\leq 7} R_{ij}f_{ij}+\frac{2}{3}\alpha|\mathbf{T}|^{2}\Delta f+4\alpha\sum_{1\leq i,j\leq 7}\widehat{\mathbf{T}}_{ij}f_{ij}\Bigg)\notag.
    \end{align}
Since $\Delta f=\partial_{t}f-|\nabla f|^{2}$, it follows that
    \begin{equation}
        \alpha\partial_{t}(\Delta f)=\alpha f_{tt}-\alpha\Bigg(\sum_{1\leq i,j\leq 7}\left(2R_{ij}f_{i}f_{j}+4\widehat{\mathbf{T}}_{ij}f_{i}f_{j}\right)+\frac{2}{3}|\mathbf{T}|^{2}|\nabla f|^{2}+2\nabla f\cdot\nabla (\partial_{t}f)\Bigg)\notag,
    \end{equation}
and
    \begin{align}
        \Delta F&=t\Bigg(2\sum_{1\leq i,j\leq 7}f_{ij}^{2}+2\nabla f\cdot\nabla\Delta f+2\sum_{1\leq i,j\leq 7}R_{ij}f_{i}f_{j}-\alpha f_{tt}+2\alpha \sum_{1\leq i,j\leq 7}R_{ij}f_{i}f_{j}\notag\\
        &\quad+\frac{2}{3}\alpha|{\mathbf{T}}|^{2}|\nabla f|^{2}+4\alpha\sum_{1\leq i,j\leq 7}\widehat{\mathbf{T}}_{ij}f_{i}f_{j}+2\alpha\nabla f\cdot\nabla (\partial_{t}f)\notag\\ 
        &\quad+2\alpha\sum_{1\leq i,j\leq 7} R_{ij}f_{ij}+\frac{2}{3}\alpha|\mathbf{T}|^{2}\Delta f+4\alpha\sum_{1\leq i,j\leq 7}\widehat{\mathbf{T}}_{ij}f_{ij}\Bigg)\notag.
    \end{align}
On the other hand, we have
    \begin{eqnarray*}
        \partial_{t}F&=&|\nabla f|^{2}-\alpha \partial_{t}f+t\Bigg(2\sum_{1\leq i,j\leq 7}R_{ij}f_{i}f_{j}+\frac{2}{3}|\mathbf{T}|^{2}|\nabla f|^{2}\\
        &&+4\sum_{1\leq i,j\leq 7}\widehat{\mathbf{T}}_{ij}f_{i}f_{j}+2\nabla f\cdot\nabla (\partial_{t}f)\Bigg)-t\alpha f_{tt}.\notag
    \end{eqnarray*}
Now we obtain
    \begin{align}\label{tri-par F}
        \left(\Delta-\partial_{t}\right)F&=t\Big(2\nabla f\cdot\nabla\Delta f+2\alpha\nabla f\cdot\nabla (\partial_{t}f)-2\nabla f\cdot\nabla (\partial_{t}f)\Big)\\
        &\quad+t\Bigg(2\sum_{1\leq i,j\leq 7}f_{ij}^{2}+2\alpha \sum_{1\leq i,j\leq 7}R_{ij}f_{ij}+4\alpha\sum_{1\leq i,j\leq 7}\widehat{\mathbf{T}}_{ij}f_{ij}\Bigg)\notag\\
        &\quad+t\bigg(\frac{2}{3}\alpha|\mathbf{T}|^{2}|\nabla f|^{2}+\frac{2}{3}\alpha|\mathbf{T}|^{2}\Delta f-\frac{2}{3}|{\mathbf{T}}|^{2}|\nabla f|^{2}\bigg)\notag\\
        &\quad+t\Bigg(2\alpha \sum_{1\leq i,j\leq 7}R_{ij}f_{i}f_{j}+4\alpha\sum_{1\leq i,j\leq 7}\widehat{\mathbf{T}}_{ij}f_{i}f_{j}-4\sum_{1\leq i,j\leq 7}\widehat{\mathbf{T}}_{ij}f_{i}f_{j}\Bigg)\notag\\
        &\quad-\Big(|\nabla f|^{2}-\alpha \partial_{t}f\Big).\notag
    \end{align}
Again using $\Delta f=\partial_{t}f-|\nabla f|^{2}$, the first and third line in the right side of $(\ref{tri-par F})$ become
    \begin{align}
        &\quad t\Big(2\nabla f\cdot\nabla\Delta f+2\alpha\nabla f\cdot\nabla (\partial_{t}f)-2\nabla f\cdot\nabla (\partial_{t}f)\Big)\notag\\
        &=2t\nabla f\cdot\nabla\Big(\partial_{t}f-|\nabla f|^{2}+\alpha \partial_{t}f-\partial_{t}f\Big)=-2\nabla f\cdot\nabla F\notag,\\
        &\quad t\Big(\frac{2}{3}\alpha|\mathbf{T}|^{2}|\nabla f|^{2}+\frac{2}{3}\alpha|\mathbf{T}|^{2}\Delta f-\frac{2}{3}|\mathbf{T}|^{2}|\nabla f|^{2}\Big)\notag\\
        &=\frac{2}{3}t|\mathbf{T}|^{2}\Big(\alpha|\nabla f|^{2}+\alpha\Delta f-|\nabla f|^{2}\Big)=-\frac{2}{3}|\mathbf{T}|^{2}F.\notag
    \end{align}
For the second line, using the trick in \cite{heat equation under RF by Cao xiaodong, heat equation under RHF}, we get
    \begin{align}
        &\quad \sum_{1\leq i,j\leq 7}\left(f_{ij}^{2}+\alpha R_{ij}f_{ij}+2\alpha\widehat{\mathbf{T}}_{ij}f_{ij}\right)\notag\\
        &=\sum_{1\leq i,j\leq 7}\left((a\alpha+2b\alpha)f_{ij}^{2}+\alpha R_{ij}f_{ij}+2\alpha\widehat{\mathbf{T}}_{ij}f_{ij}\right)\notag\\
        &=\sum_{1\leq i,j\leq 7}\left(a\alpha f_{ij}^{2}+\alpha\left|\sqrt{b}f_{ij}+\frac{R_{ij}}{2\sqrt{b}}\right|^{2}-\frac{\alpha}{4b}|{\rm Ric}|^{2}+\alpha\left|\sqrt{b}f_{ij}+\frac{\widehat{\mathbf{T}}_{ij}}{\sqrt{b}}\right|^{2}-\frac{\alpha}{b}|\widehat{\mathbf{T}}|^{2}\right)\notag\\
        &\geq a\alpha\sum_{1\leq i,j\leq 7} f_{ij}^{2}-\frac{\alpha}{4b}|{\rm Ric}|^{2}-\frac{\alpha}{b}|\widehat{\mathbf{T}}|^{2}\notag,
    \end{align}
where $a, b$ are constants satisfying $\displaystyle{a+2b=\frac{1}{\alpha}}$. Noting that 
    $$
    \sum_{1\leq i,j\leq 7}f_{ij}^{2}\geq\frac{1}{7}\Bigg(\sum_{1\leq i\leq 7} f_{ii}\Bigg)^{2}=\frac{(\Delta f)^{2}}{7},\quad |{\rm Ric}|^{2}\leq 7K^{2},\quad|\widehat{\mathbf{T}}|^{2}\leq 49
K^{2},$$
so it becomes 
    \begin{equation}
        \sum_{1\leq i,j\leq 7}\left(f_{ij}^{2}+\alpha R_{ij}f_{ij}+2\alpha\widehat{\mathbf{T}}_{ij}f_{ij}\right)\geq\frac{a\alpha(\Delta f)^{2}}{7}-\left(\frac{7\alpha}{4b}+\frac{49\alpha }{b}\right)K^{2}\notag.
    \end{equation}
For the fourth line, since $\mathbf{T}_{ij}$ is skew-symmetric, we obtain
    \begin{eqnarray*}  
    \sum_{1\leq i,j\leq 7}\widehat{\mathbf{T}}_{ij}f_{i}f_{j}&=&\sum_{1\leq i,j,k\leq 7}\mathbf{T}_{i}^{\ k}\mathbf{T}_{kj}f_{i}f_{j}\\
    &=&-\sum_{1\leq i,j,k\leq 7}(\mathbf{T}_{ik}f_{i})(\mathbf{T}_{jk}f_{j}) \ \ = \ \ -\ \Bigg|\sum_{1\leq i\leq 7}\mathbf{T}_{ik}f_{i}\Bigg|^{2}\notag.
    \end{eqnarray*}
Together with the Cauchy inequality, we have
    \begin{align}
        \sum_{1\leq i,j\leq 7}\widehat{\mathbf{T}}_{ij}f_{i}f_{j}&=-\ \Bigg|\sum_{1\leq i\leq 7}\mathbf{T}_{ik}f_{i}\Bigg|^{2}\notag\\
        &=-\sum_{1\leq k\leq 7}\Bigg(\sum_{1\leq i\leq 7}\mathbf{T}_{ik}f_{i}\Bigg)^{2}\notag\\
        &\geq-\sum_{1\leq k\leq 7}\Bigg(\sum_{1\leq i\leq 7}\mathbf{T}_{ik}^{2}|\nabla f|^{2}\Bigg)\notag\\
        &=-|\mathbf{T}|^{2}|\nabla f|^{2}=R|\nabla f|^{2}\geq -7K|\nabla f|^{2}\notag.
    \end{align}
Now, the fourth line becomes
    \begin{align}
        t\sum_{1\leq i,j\leq 7}&\left(2\alpha R_{ij}f_{i}f_{j}+4\alpha\widehat{\mathbf{T}}_{ij}f_{i}f_{j}-4\widehat{\mathbf{T}}_{ij}f_{i}f_{j}\right)\notag\\
        &\geq-2\alpha tK|\nabla f|^{2}-28t(\alpha-1)K|\nabla f|^{2}=-\Big(2\alpha+28(\alpha-1)\Big)tK|\nabla f|^{2}\notag.
    \end{align}
Substituting all these terms into $(\ref{tri-par F})$, we now obtain
    \begin{align}
        \label{tri-par F 2}\left(\Delta-\partial_{t}\right)F&\geq-2\nabla f\cdot\nabla F+\frac{2a\alpha t}{7}\Big(|\nabla f|^{2}-\partial_{t}f\Big)^{2}-\Big(2\alpha+28(\alpha-1)\Big)tK|\nabla f|^{2}\notag\\
        &\quad-\frac{2}{3}|{\bf T}|^{2}F-\Big(|\nabla f|^{2}-\alpha \partial_{t}f\Big)-\left(\frac{7\alpha}{2b}+\frac{98\alpha }{b}\right)tK^{2}.
    \end{align}
Following the trick in \cite{Gradient estimate Liu, heat equation under RF by Cao xiaodong}, set $y=|\nabla f|^{2}$, $z=\partial_{t}f$. Observe that 
    \begin{equation}
        (y-z)^{2}=\frac{1}{\alpha^{2}}(y-\alpha z)^{2}+\left(\frac{\alpha-1}{\alpha}\right)^{2}y^{2}+\frac{2(\alpha-1)}{\alpha^{2}}y(y-\alpha z)\notag,
    \end{equation}
    and $\displaystyle{mx^{2}-nx\geq-\frac{n^{2}}{4m}}$ for any $m,n>0$. Now we have 
    \begin{align}
        &\quad\frac{2a\alpha t}{7}\left(\big(|\nabla f|^{2}-f_{t}\big)^{2}-\frac{7\alpha+98(\alpha-1)}{a\alpha}K|\nabla f|^{2}\right)\notag\\
        &=\frac{2a\alpha t}{7}\left((y-z)^{2}-\frac{7\alpha+98(\alpha-1)}{a\alpha}Ky\right)\notag\\
        &=\frac{2a\alpha t}{7}\Big[\frac{1}{\alpha^{2}}(y-\alpha z)^{2}+\left(\frac{\alpha-1}{\alpha}\right)^{2}y^{2}+\frac{2(\alpha-1)}{\alpha^{2}}y(y-\alpha z)-\frac{7\alpha+98(\alpha-1)}{a\alpha}Ky\Big]\notag\\
        &\geq\frac{2a\alpha t}{7}\left(\frac{1}{\alpha^{2}}(y-\alpha z)^{2}+\frac{2(\alpha-1)}{\alpha^{2}}y(y-\alpha z)-\frac{49K^{2}[\alpha+14(\alpha-1)]^{2}}{4(\alpha-1)^{2}a^{2}}\right)\notag\\
        &=\frac{2a}{7\alpha}\frac{F^{2}}{t}+\frac{2(\alpha-1)}{\alpha^{2}}|\nabla f|^{2}\frac{F}{t}\frac{2a\alpha t}{7}-\frac{7\alpha K^{2}[\alpha+14(\alpha-1)]^{2}}{2(\alpha-1)^{2}a}t\notag\\
        &\geq\frac{2a}{7\alpha}\frac{F^{2}}{t}-\frac{7\alpha K^{2}[\alpha+14(\alpha-1)]^{2}}{2(\alpha-1)^{2}a}t\notag.
    \end{align}
Taking this term into $(\ref{tri-par F 2})$, we finally arrive at
    \begin{align}
        \label{tri-par F 3}\left(\Delta-\partial_{t}\right)F&\geq-2\nabla f\cdot\nabla F+\frac{2a}{7\alpha}\frac{F^{2}}{t}-\frac{7\alpha[\alpha+14(\alpha-1)]^{2}}{2(\alpha-1)^{2}a}tK^{2}\\
        &\quad+\frac{2}{3}RF-\frac{F}{t}-\left(\frac{7\alpha}{2b}+\frac{98\alpha}{b}\right)tK^{2}\notag.
    \end{align}
We assume that $F(x,t)$ takes its maximum at $(x_{0},t_{0})$, which means $$\nabla F(x_{0},t_{0})=0, \ \ \partial_{t}F(x_{0},t_{0})\geq0, \ \ \triangle F(x_{0},t_{0})\leq0.$$
Thus, at $(x_{0},t_{0})$, we have
    \begin{equation}
        \frac{2a}{7\alpha}F^{2}-\left(1-\frac{2}{3}Rt\right)F-\frac{7\alpha [\alpha+14(\alpha-1)]^{2}}{2(\alpha-1)^{2}a}t^{2}K^{2}-\left(\frac{7\alpha}{2b}+\frac{98\alpha}{b}\right)t^{2}K^{2}\leq0.
    \end{equation}
According to the quadratic formula
    \begin{align}
        F&\leq\frac{7\alpha}{4a}\cdot\Bigg(\sqrt{\left(1-\frac{2}{3}Rt\right)^{2}+\frac{8a}{7\alpha}\bigg(\frac{7\alpha [\alpha+14(\alpha-1)]^{2}}{2(\alpha-1)^{2}a}+\left(\frac{7\alpha}{2b}+\frac{98\alpha}{b}\right)\bigg)t^{2}K^{2}}\notag\\
        &\quad+1-\frac{2}{3}Rt\Bigg)\notag,
    \end{align}
which implies
    \begin{equation}
        F\leq\frac{7\alpha}{4a}\left(2-\frac{4}{3}Rt+\frac{2[\alpha+14(\alpha-1)]}{\alpha-1}tK+\sqrt{8a\left(\frac{1}{2b}+\frac{14}{b}\right)}tK\right)\notag.
    \end{equation}
Since $F$ takes its maximum at $(x_{0},t_{0})$, for all $(x,t)\in M\times(0,T]$,
    \begin{equation}
        F(x,t)\leq F(x_{0},t_{0})\leq \frac{7\alpha}{2a}+\left(\frac{49\alpha}{3a}+\frac{7\alpha[\alpha+14(\alpha-1)]}{2a(\alpha-1)}+\frac{7\alpha}{4a}\sqrt{8a\left(\frac{1}{2b}+\frac{14}{b}\right)}\right)tK\notag.
    \end{equation}
According to the definition of $F(x,t)$, we obtain the desired result
    \begin{equation}
        \frac{|\nabla u|^{2}}{u^{2}}-\alpha\frac{\partial_{t}u}{u}\leq \frac{7\alpha}{2at}+\left(\frac{49\alpha}{3a}+\frac{105\alpha^{2}-98\alpha}{2a(\alpha-1)}+\frac{7\sqrt{29}\alpha}{2\sqrt{ab}}\right)K,
    \end{equation}
where $\alpha>1$, $\displaystyle{a+2b=\frac{1}{\alpha}}$.
\end{proof}

\begin{remark}
    For example, if we take $\displaystyle{a=2b=\frac{1}{2\alpha}}$, then the estimate becomes 
    \begin{equation}
        \frac{|\nabla u|^{2}}{u^{2}}-\alpha\frac{\partial_{t}u}{u}\leq\frac{7\alpha^{2}}{t}+\left[\left(\frac{98}{3}+7\sqrt{58}\right)\alpha^{2}+\frac{105\alpha^{3}-98\alpha^{2}}{\alpha-1}\right]K\notag,
    \end{equation}
  where $\alpha>1$.
\end{remark}

\begin{corollary}
    Let $(M,\varphi(t))_{t\in(0,T]}$ be the solution of the Laplacian $G_{2}$ flow $\eqref{G2 flow}$ on a closed $7$-dimensional manifold $M$ with $T<+\infty$ and $-Kg(t)\leq {\rm Ric}(g(t))\leq0$, where $g(t)$ is the Riemannian metric associated with $\varphi(t)$ and $K$ is a positive constant. If $u(t)$ is a positive solution of the heat equation, then on $M\times(0, T]$ such that $t_{1}<t_{2}$, we have 
    $$
    u(x,t_{1})\leq u(y,t_{2})\left(\frac{t_{2}}{t_{1}}\right)^{\frac{7}{2a}}\exp\left\{\int_{0}^{1}\left[\frac{\alpha|\gamma^{\prime}(s)|_{\sigma(s)}^{2}}{4(t_{2}-t_{1})}+(t_{2}-t_{1})C_{a,b,\alpha}K\right]ds\right\},
    $$
    where $\alpha>1$, 
    $$
    C_{a,b,\alpha}=\frac{49}{3a}+\frac{105\alpha-98}{2a(\alpha-1)}+\frac{7\sqrt{29}}{2\sqrt{ab}},
    $$
    $a,b>0$, $\displaystyle{a+2b=\frac{1}{\alpha}}$, $\gamma(s)$ is a geodesics curve connecting $x$ and $y$ with $\gamma(0)=y$ and $\gamma(1)=x$, and $|\gamma^{\prime}(s)|_{\sigma(s)}$ is the length of the vector $\gamma^{\prime}(s)$ at $\sigma(s)=(1-s)t_{2}+st_{1}$.
\end{corollary}

\begin{proof}
    We can write Li-Yau type gradient estimate in Theorem \ref{theorem 3.1} as follows:
    \begin{align}
        \frac{|\nabla u|^{2}}{u^{2}}-\alpha\frac{\partial_{t}u}{u}\leq\frac{C_{\alpha}}{t}+C_{\alpha,a,b}K\notag,
    \end{align}
    where
    $$
    C_{\alpha}=\frac{7\alpha}{2a}, \ \ \ C_{\alpha,a,b}=\frac{49\alpha}{3a}+\frac{105\alpha^{2}-98\alpha}{2a(\alpha-1)}+\frac{7\sqrt{29}\alpha}{2\sqrt{ab}}.
    $$

    Choosing a geodesics curve $\gamma(s)$ connects $x$ and $y$ with $\gamma(0)=y$ and $\gamma(1)=x$. We define $l(s)=\ln u(\gamma(s),(1-s)t_{2}+st_{1})$ and $\sigma(s)=(1-s)t_{2}+st_{1}$, then we have $l(0)=\ln u(y,t_{2})$ and $l(1)=\ln u(x,t_{1})$. By calculating, we have
    \begin{align}
        \label{l(s)}\frac{\partial l(s)}{\partial s}&=(t_{2}-t_{1})\left(\frac{\nabla u}{u}\frac{\gamma^{\prime}(s)}{(t_{2}-t_{1})}-\frac{u_{t}}{u}\right)\\
        &\leq(t_{2}-t_{1})\left(\frac{\alpha|\gamma^{\prime}(s)|_{\sigma(s)}^{2}}{4(t_{2}-t_{1})^{2}}+\frac{|\nabla u|^{2}}{\alpha u^{2}}-\frac{u_{t}}{u}\right)\notag\\
        &\leq \frac{\alpha|\gamma^{\prime}(s)|_{\sigma(s)}^{2}}{4(t_{2}-t_{1})}+\frac{t_{2}-t_{1}}{\alpha}\left(\frac{C_{\alpha}}{\sigma(s)}+C_{\alpha,a,b}K\right)\notag,
    \end{align}
    where $|\gamma^{\prime}(s)|_{\sigma(s)}$ is the length of the vector $\gamma^{\prime}(s)$ at $\sigma(s)$. Integrating $\eqref{l(s)}$ over $\gamma(s)$, we get
    \begin{align}
        \ln\frac{u(x,t_{1})}{u(y,t_{2})}&=\int_{0}^{1}\frac{\partial l(s)}{\partial s}ds\notag\\
        &\leq \int_{0}^{1} \left[\frac{\alpha|\gamma^{\prime}(s)|_{\sigma(s)}^{2}}{4(t_{2}-t_{1})}+\frac{t_{2}-t_{1}}{\alpha}\left(\frac{C_{\alpha}}{\sigma(s)}+C_{\alpha,a,b}K\right)\right]ds\notag\\
        &=\int_{0}^{1}\left[\frac{\alpha|\gamma^{\prime}(s)|_{\sigma(s)}^{2}}{4(t_{2}-t_{1})}+\frac{t_{2}-t_{1}}{\alpha}C_{\alpha,a,b}K\right]ds+\frac{7}{2a}\ln\frac{t_{2}}{t_{1}}\notag.
    \end{align}
Thus, we get the desired result.    
\end{proof}

\begin{theorem}\label{theorem 3.4}
    Let $(M,\varphi(t))_{t\in(0,T]}$ be the solution of the Laplacian $G_{2}$ flow $\eqref{G2 flow}$ on a closed $7$-dimensional manifold $M$ with $T<+\infty$ and $-Kg(t)\leq {\rm Ric}(g(t))\leq0$, where $g(t)$ is the Riemannian metric associated with $\varphi(t)$ and $K$ is a positive constant. If $u(t)$ is a positive solution of the heat equation, then on $M\times(0, T]$ the following estimate
    \begin{equation}
        |\nabla_{g(t)} u(t)|^{2}_{g(t)}\leq\frac{u(t)}{t}\left(u(t)\ln\frac{A}{u(t)}+\lambda A^{2}-\lambda\eta^{2}\right)
    \end{equation}
holds, where $\displaystyle{\eta=\min_{M}u(0)}$, $\displaystyle{A=\max_{M}u(0)}$ and $\lambda$ is a constant depending on $K$, $\eta$ and $T$.
\end{theorem}
\begin{proof}
    First, using the Bochner technique with the flow equation, we have that 
    \begin{equation}
        (\partial_{t}-\Delta)|\nabla u|^{2}=-2|\nabla^{2}u|^{2}+\frac{2}{3}|\mathbf{T}|^{2}|\nabla u|^{2}-4\sum_{1\leq k\leq 7}\Bigg|\sum_{1\leq i\leq 7}\mathbf{T}_{ik}\nabla_{i}u\Bigg|^{2}.
    \end{equation}
Let $\lambda$ be a constant to be fixed later. Setting
$$
P=t\frac{|\nabla u|^{2}}{u}-u\ln\frac{A}{u}+\lambda u^{2},
$$
some calculations show that 
    \begin{align}
        \partial_{t}\bigg(\frac{|\nabla u|^{2}}{u}\bigg)&=\frac{\partial_{t}|\nabla u|^{2}}{u}-\frac{|\nabla u|^{2}}{u^{2}}\partial_{t}u\notag,\\
       \Delta\bigg(\frac{|\nabla u|^{2}}{u}\bigg)&=\frac{\Delta|\nabla u|^{2}}{u}+\Delta\left(\frac{1}{u}\right)|\nabla u|^{2}+2\nabla\left(\frac{1}{u}\right)\cdot\nabla(|\nabla u|^{2})\notag\\
        &=\frac{\Delta|\nabla u|^{2}}{u}-\left(\frac{\Delta u}{u^{2}}-2\frac{|\nabla u|^{2}}{u^{3}}\right)|\nabla u|^{2}-\frac{4}{u^{2}}\sum_{1\leq i,j\leq 7}u_{ij}u_{i}u_{j}\notag,\\
        \partial_{t}\left(u\ln\frac{A}{u}\right)&=\Delta\left(u\ln\frac{A}{u}\right)+\frac{|\nabla u|^{2}}{u}\notag,\\
        \partial_{t}\big(u^{2}\big)&=\Delta \big(u^{2}\big)-2|\nabla u|^{2}.\notag
    \end{align}
Combining these equations we obtain
    \begin{align}
        \left(\partial_{t}-\Delta\right)P&=t\left(\partial_{t}\frac{|\nabla u|^{2}}{u}-\Delta\frac{|\nabla u|^{2}}{u}\right)-2\lambda|\nabla u|^{2}\notag\\
        &=t\Bigg(\frac{(\partial_{t}-\Delta)|\nabla u|^{2}}{u}-2\frac{|\nabla u|^{4}}{u^{3}}+\frac{4}{u^{2}}\sum_{1\leq i,j\leq 7}u_{ij}u_{i}u_{j}\Bigg)-2\lambda|\nabla u|^{2}\notag\\
        &=\frac{t}{u}\Bigg(-2|\nabla^{2}u|^{2}+\frac{2}{3}|\mathbf{T}|^{2}|\nabla u|^{2}-4\sum_{1\leq k\leq 7}\Bigg|\sum_{1\leq i\leq 7}\mathbf{T}_{ik}u_{i}\Bigg|^{2}\Bigg)\notag\\
        &\quad-\frac{2t}{u}\Bigg(\frac{|\nabla u|^{4}}{u^{2}}-\frac{2}{u}\sum_{1\leq i,j\leq 7}u_{ij}u_{i}u_{j}\Bigg)-2\lambda|\nabla u|^{2}\notag\\
        &=-\frac{2t}{u}\sum_{1\leq i,j\leq 7}\left|u_{ij}-\frac{u_{i}u_{j}}{u}\right|^{2}+\left(\frac{2t}{3u}|\mathbf{T}|^{2}-2\lambda\right)|\nabla u|^{2}\notag\\
        &\quad-\frac{4t}{u}\sum_{1\leq k\leq 7}\Bigg|\sum_{1\leq i\leq 7}\mathbf{T}_{ik}u_{i}\Bigg|^{2}\notag.
    \end{align}
Since $\eta\leq u(0)\leq A$ and $|\mathbf{T}|^{2}=-R$, taking $\displaystyle{\lambda\geq\frac{7KT}{3\eta}}$, we obtain $\left(\partial_{t}-\Delta\right)P\leq0$.

According to the maximum principle, we obtain
    \begin{equation}
       P(t)\leq\max_{M}P(0)=\lambda A^{2}\notag,
    \end{equation}
which means 
    \begin{equation}
        |\nabla u|^{2}\leq\frac{u}{t}\left(u\ln\frac{A}{u}+\lambda A^{2}-\lambda u^{2}\right),
    \end{equation}
where $\lambda$ is a constant depending on $K$, $\eta$ and $T$. Thus we get the desired result.
\end{proof}

\section{Parabolic frequency on Laplacian $G_{2}$ flow with bounded Ricci curvature }\label{section 4}

In this section, using the Li-Yau type gradient estimate and Hamilton type gradient estimate, we study the parabolic frequency for the solution of the heat equation $\eqref{heat equation}$
under the Laplacian $G_{2}$ flow $\eqref{G2 flow}$ with bounded Ricci curvature.

For a time-dependent function $u=u(t):M\times[t_{0},t_{1}]\rightarrow \mathbb{R}^{+}$ with $u(t),\partial_{t}u(t)\in W_{0}^{2,2}(d\mu_{g(t)})$ and for all $t\in[t_{0},t_{1}]\subset(0,T)$, we define 
\begin{align}
I(t)&=\int_{M} u^{2}(t) d \mu_{g(t)},\notag\\
D(t)&=h(t)\int_{M}|\nabla_{g(t)}u(t)|_{g(t)}^{2}d \mu_{g(t)}=-h(t)\int_{M}\langle u(t),\Delta_{g(t),f(t)}u(t)\rangle_{g(t)} d\mu_{g(t)},\notag\\
U(t)&=\exp\left\{-\int_{t_{0}}^{t}\left(\frac{h'(s)}{h(s)}-\frac{2}{3}R_{0}+\frac{28+C_{1}(A,\eta)}{s}+cK+\frac{7}{2}C(s)\right)ds\right\}\frac{D(t)}{I(t)}\notag
\end{align}
where $h(t)$ is a time-dependent function, $K$ and $c$ are both positive constants, 
$$
R_{0}=\min_{M\times[t_{0},t_{1}]}R(t), \ \ \ C_{1}(A,\eta)=\ln\frac{A}{\eta}+\lambda\frac{A^{2}}{\eta}, \ \ \ C(t)=\frac{C_{1}(A,\eta)}{t}, \ \ \ 
$$
and $\lambda$ is the constant in Theorem \ref{theorem 3.4},
$$
\eta=\min_{M}u(0), \ \ \ A=\max_{M}u(0). 
$$
Observe that, $A$ and $\eta$ are both positive constants.

\begin{lemma}\label{lemma 4.1}
    Under the Laplacian $G_{2}$ flow $\eqref{G2 flow}$, the norm of the gradient of any smooth function $u(t)$ satisfies the following equation
\begin{eqnarray}
    (\partial_{t}-\Delta)|\nabla u|^{2}&=&-2|\nabla^{2}u|^{2}+2\langle\nabla u,\nabla(\partial_{t}-\Delta)u\rangle+\frac{2}{3}|{\bf T}|^{2}|\nabla u|^{2}\\
    &&- \ 4\sum_{1\leq k\leq 7}\bigg|\sum_{1\leq i\leq 7}{\bf T}_{ik}\nabla_{i}u\bigg|^{2}.\notag
\end{eqnarray}
\end{lemma}
\begin{proof}
    At first, note that 
\begin{equation}
    \partial_{t}|\nabla u|^{2}=2{\rm Ric}(\nabla u,\nabla u)+\frac{2}{3}|{\bf T}|^{2}|\nabla u|^{2}+4\sum_{1\leq i,j\leq 7}\widehat{{\bf T}}_{ij}\nabla_{i}u\nabla_{j}u+2\langle\nabla u,\nabla\partial_{t}u\rangle
\end{equation}
Since ${\bf T}_{ij}$ is anti-symmetric, we obtain
\begin{eqnarray}
    \sum_{1\leq i,j\leq 7}\widehat{{\bf T}}_{ij}\nabla_{i}u\nabla_{j}u&=&\sum_{1\leq i,j, m\leq 7}{\bf T}_{im}{\bf T}^{\ m}_{j}\nabla_{i}u\nabla_{j}u\nonumber\\
    &=&-\sum_{1\leq i,j, m\leq 7}({\bf T}_{im}\nabla_{i}u)({\bf T}_{m}^{\ j}\nabla_{j}u)\\
    &=&-\sum_{1\leq k\leq 7}\bigg|\sum_{1\leq i\leq 7}{\bf T}_{ik}\nabla_{i}u\bigg|^{2}.\nonumber
\end{eqnarray}
Together with the Bochner formula, we obtain the desired result.
\end{proof}

\begin{theorem}\label{theorem 4.2}
Let $(M,\varphi(t))_{t\in(0,T]}$ be the solution of the Laplacian $G_{2}$ flow $\eqref{G2 flow}$ on a closed $7$-dimensional manifold $M$ with $T<+\infty$ and $-Kg(t)\leq {\rm Ric}(g(t))\leq0$, where $g(t)$ is the Riemannian metric associated with $\varphi(t)$ and $K$ is a positive constant. If $u(t)$ is a positive solution of the heat equation $\eqref{heat equation}$ with $\eta\leq u(0)\leq A$, then the following holds.

\begin{itemize} 

\item[(i)] If $h(t)$ is a negative time-dependent function, then the parabolic frequency $U(t)$ is monotone increasing along the Laplacian $G_{2}$ flow.

\item[(ii)] If $h(t)$ is a positive time-dependent function, then the parabolic frequency $U(t)$ is monotone decreasing along the Laplacian $G_{2}$ flow.
\end{itemize}
\end{theorem}
\begin{proof}
Before discussing the monotonicity of $U(t)$, we need to calculate the derivative of $I(t)$ and $D(t)$. 
Taking $\displaystyle{\alpha=2,a=\frac{1}{4},b=\frac{1}{8}}$ in Theorem \ref{theorem 3.1}, we obtain
\begin{equation}
        \frac{|\nabla u|^{2}}{u^{2}}-2\frac{\partial_{t} u}{u}\leq \frac{28}{t}+cK\label{gradient estiamte},
\end{equation}
where $\displaystyle{c=\frac{392}{3}+448+28\sqrt{58}}$. Then we can get the derivative of $I(t)$,
\begin{align}
I'(t)&=\frac{d}{dt}\left(\int_{M}u^{2} d\mu\right)\label{I'}\\
&=2\int_{M}\left(u\cdot\partial_{t}u-|\nabla u|^{2}\right)d \mu -2\int_{M}u\Delta u\!\ d\mu\notag\\
&=2\int_{M}\left(u\cdot\partial_{t}u-\frac{1}{2}|\nabla u|^{2}\right)d \mu -2\int_{M}u\Delta u \!\ d\mu-\int_{M}|\nabla u|^{2}d\mu\notag\\
&\geq-\left(\frac{28}{t}+cK\right)I(t)-2\int_{M}u\Delta u \!\ d\mu-\int_{M}|\nabla u|^{2}d\mu\notag\\
&\geq-\left(\frac{28+C_{1}(A,\eta)}{t}+cK+\frac{7C(t)}{2}\right)I(t)-\frac{2}{7C(t)}\int_{M}|\Delta u|^{2}d\mu.\notag
\end{align}
where $\displaystyle{C_{1}(A,\eta)=\ln\frac{A}{\eta}+\lambda\frac{A^{2}}{\eta}}$, and we use Young's inequality and Theorem \ref{theorem 3.4} in the last line.

For the derivative of $D(t)$, according to Lemma \ref{lemma 4.1}, we have
\begin{align}
D'(t)&=h'(t)\int_{M}|\nabla u|^{2}d \mu+h(t)\frac{d}{dt}\left(\int_{M}|\nabla u|^{2}d\mu\right)\label{D'}\\
&=h'(t)\int_{M}|\nabla u|^{2}d \mu+h(t)\int_{M}(\partial_{t}-\Delta)|\nabla u|^{2}d\mu\notag\\
&=h'(t)\int_{M}|\nabla u|^{2}d\mu-2h(t)\int_{M}|\nabla^{2} u|^{2}d\mu\notag\\
&\quad+\frac{2}{3}h(t)\int_{M}|{\bf T}|^{2}|\nabla u|^{2}d\mu-4h(t)\int_{M}|{\bf T}_{ik}\nabla^{i}u|^{2}d\mu\notag.
\end{align}
If $h(t)< 0$, then by $(\ref{D'})$,
\begin{align}
D'(t)\geq\left( h^{\prime}-\frac{2}{3}hR_{0}\right)\int_{M}|\nabla u|^{2}d\mu-2h\int_{M}|\nabla^{2} u|^{2}d\mu\notag,
\end{align}
together with $(\ref{I'})$ and Theorem \ref{theorem 3.4}, yields
\begin{align}
I^{2}(t)U'(t)&\geq \exp\left\{-\int_{t_{0}}^{t}\left(\frac{h'(s)}{h(s)}-\frac{2}{3}R_{0}+\frac{28+C_{1}(A,\eta)}{s}+cK+\frac{7C(s)}{2}\right) ds\right\}\notag\\
&\quad\cdot\left[-2h I(t)\left(\int_{M}|\nabla^{2}u|^{2}d\mu\right)+\frac{2h}{7C(t)}\left(\int_{M}|\Delta u|^{2}d\mu\right)\left(\int_{M}|\nabla u|^{2}d\mu\right)\right]\notag\\
&\geq \exp\left\{-\int_{t_{0}}^{t}\left(\frac{h'(s)}{h(s)}-\frac{2}{3}R_{0}+\frac{28+C_{1}(A,\eta)}{s}+cK+\frac{7C(s)}{2}\right) ds\right\}\notag\\
&\quad\cdot\left[-\frac{2h}{7}I(t)\left(\int_{M}|\Delta u|^{2}d\mu\right)+\frac{2h}{7C(t)}\left(\int_{M}|\Delta u|^{2}d\mu\right)\left(\int_{M}|\nabla u|^{2}d\mu\right)\right]\notag\\
&\geq \exp\left\{-\int_{t_{0}}^{t}\left(\frac{h'(s)}{h(s)}-\frac{2}{3}R_{0}+\frac{28+C_{1}(A,\eta)}{s}+cK+\frac{7C(s)}{2}\right) ds\right\}\notag\\
&\quad\cdot\left[-\frac{2h}{7}I(t)\left(\int_{M}|\Delta u|^{2}d\mu\right)+\frac{2h}{7C(t)}\cdot\frac{C_{1}(A,\eta)}{t}I(t)\left(\int_{M}|\Delta u|^{2}d\mu\right)\right]\notag\\
&=0\notag
\end{align}
where we take trace over $|\nabla^{2}u|^{2}$ and let $\displaystyle{C(t)=\frac{C_{1}(A,\eta)}{t}}$.

On the other hand, if $h(t)>0$, similarly, we have
\begin{align}
I^{2}(t)U'(t)&\leq \exp\left\{-\int_{t_{0}}^{t}\left(\frac{h'(s)}{h(s)}-\frac{2}{3}R_{0}+\frac{28+C_{1}(A,\eta)}{s}+cK+\frac{7C(s)}{2}\right) ds\right\}\notag\\
&\quad\cdot\left[-2hI(t)\left(\int_{M}|\nabla^{2}u|^{2}d\mu\right)+\frac{2h}{7C(t)}\left(\int_{M}|\Delta u|^{2}d\mu\right)\left(\int_{M}|\nabla u|^{2}d\mu\right)\right]\notag\\
&\leq \exp\left\{-\int_{t_{0}}^{t}\left(\frac{h'(s)}{h(s)}-\frac{2}{3}R_{0}+\frac{28+C_{1}(A,\eta)}{s}+cK+\frac{7C(s)}{2}\right) ds\right\}\notag\\
&\quad\cdot\left[-\frac{2h}{7}I(t)\left(\int_{M}|\Delta u|^{2}d\mu\right)+\frac{2h}{7C(t)}\cdot\frac{C_{1}(A,\eta)}{t}I(t)\left(\int_{M}|\Delta u|^{2}d\mu\right)\right]\notag\\
&=0.\notag
\end{align}
Thus we get the desired result.
\end{proof}

We define the first nonzero eigenvalue of the Laplacian $G_{2}$ flow $(M,\varphi(t))_{t\in (0,T]}$ with the weighted measure $d\mu_{g(t)}$ by
$$\lambda_{M}(t)=\inf\left\{\left.\frac{\displaystyle{\int_{M}|\nabla_{g(t)}u|_{g(t)}^{2}d \mu_{g(t)}}}{\displaystyle{\int_{M}u^{2}d \mu_{g(t)}}}\right| 0<u\in C^{\infty}(M)\setminus \{0\}\right\}.$$
Then we have the following corollary by Theorem \ref{theorem 4.2}.

\begin{corollary}
Let $(M,\varphi(t))_{t\in(0,T]}$ be the solution of the Laplacian $G_{2}$ flow $\eqref{G2 flow}$ on a closed $7$-dimensional manifold $M$ with $T<+\infty$ and $-Kg(t)\leq {\rm Ric}(g(t))\leq0$, where $g(t)$ is the Riemannian metric associated with $\varphi(t)$ and $K$ is a positive constant. If $u(t)$ is a positive solution of the heat equation $\eqref{heat equation}$ with $\eta\leq u(0)\leq A$, then for any $t\in[t_{0},t_{1}]\subset(0,T)$, the following holds.

\begin{itemize}
    \item[(i)]  If $h(t)$ is a negative time-dependent function, then $\beta(t)h(t)\lambda_{M}(t)$ is a monotone increasing function.
    \item[(ii)] If $h(t)$ is a positive time-dependent function, then $\beta(t)h(t)\lambda_{M}(t)$ is a monotone decreasing function.
\end{itemize}
where
$$
\beta(t)=\exp\left\{-\int_{t_{0}}^{t}\left(\frac{h'(s)}{h(s)}-\frac{2}{3}R_{0}+\frac{28+C_{1}(A,\eta)}{s}+cK+\frac{7C(s)}{2}\right) ds\right\}.
$$
\end{corollary}

\begin{corollary}
Let $(M,\varphi(t))_{t\in(0,T]}$ be the solution of the Laplacian $G_{2}$ flow $\eqref{G2 flow}$ on a closed $7$-dimensional manifold $M$  with $T<+\infty$ and $-Kg(t)\leq {\rm Ric}(g(t))\leq0$, where $g(t)$ is the Riemannian metric associated with $\varphi(t)$ and $K$ is a positive constant. If $u(t)$ is a positive solution of the heat equation $\eqref{heat equation}$ with $\eta\leq u(0)\leq A$, then for any $t\in[t_{0},t_{1}]\subset(0,T)$, 
$$I(t_{1})\geq\exp\left\{2U(t_{0})\int_{t_{0}}^{t_{1}}\frac{dt}{-h(t)\beta(t)}\right\}I(t_{0}).$$
\end{corollary}

\begin{proof}
We give the proof of case $h(t)<0$ (The case $h(t)>0$ is similar to it). According to the definition of $U(t)$, yields
\begin{align}
    \label{4.7}\frac{d}{dt}\ln(I(t))=\frac{I^{\prime}(t)}{I(t)}=-\frac{2D(t)}{h(t)I(t)}=\frac{2U(t)}{-h(t)\beta(t)}.
\end{align}
By Theorem \ref{theorem 4.2}, integrating $\eqref{4.7}$ from $t_{0}$ to $t_{1}$, we get
\begin{align}
    \ln I(t_{1})-\ln I(t_{0})=2\int_{t_{0}}^{t_{1}}\frac{U(t)}{-h(t)\beta(t)}dt\geq 2U(t_{0})\int_{t_{0}}^{t_{1}}\frac{dt}{-h(t)\beta(t)}.\notag
\end{align}
From the boundedness of time-dependent function $h(t)$, we have
$$I(t_{1})\geq\exp\left\{2U(t_{0})\int_{t_{0}}^{t_{1}}\frac{dt}{-h(t)\beta(t)}\right\}I(t_{0}).$$
We prove this corollary.
\end{proof}

\section{Parabolic frequency on Laplacian $G_{2}$ flow with bounded Bakry-\'{E}mery Ricci curvature }\label{section 5}

In this section, we study the parabolic frequency for the solution of the linear equation $\eqref{lpe}$ and the more general equations under the Laplacian $G_{2}$ flow $\eqref{G2 flow}$ with bounded Bakry-\'{E}mery Ricci curvature.

For a time-dependent function $u=u(t):M\times[t_{0},t_{1}]\rightarrow \mathbb{R}$ with $u(t), \partial_{t}u(t)\in W^{2,2}_{0}(d\mu_{g(t)})$ for all $t\in[t_{0},t_{1}]\subset(0,T)$, we denote by
\begin{align}
I(t)&=\int_{M} u^{2}(t) d \mu_{g(t)},\notag\\
D(t)&=h(t)\int_{M}|\nabla_{g(t)}u(t)|_{g(t)}^{2}d \mu_{g(t)}\notag\\
&=-h(t)\int_{M}\langle u(t),\Delta_{g(t),f(t)}u(t)\rangle_{g(t)} d\mu_{g(t)},\notag\\
U(t)&=\exp\left\{-\int_{t_{0}}^{t}\left(-\frac{2}{3}R_{0}+\frac{h'(s)+\kappa(s)}{h(s)}\right) ds\right\}\frac{D(t)}{I(t)},\notag
\end{align}
where $\displaystyle{R_{0}=\min_{M\times[t_{0},t_{1}]}R(t)}$, $h(t)$ and $\kappa(t)$ are both time-dependent smooth functions.

\subsection{Parabolic frequency for the linear heat equation under Laplacian $G_{2}$ flow.}
In this section, we consider the parabolic frequency $U(t)$ for the solution of the linear heat equation 
\begin{equation}
    (\partial_{t}-\Delta_{g(t)})u(t)=a(t)u(t)\label{lpe}
\end{equation}
under the Laplacian $G_{2}$ flow $\eqref{G2 flow}$, where $a(t)$ is a time-dependent smooth function. 
At first, we give some lemmas.

\begin{lemma}\label{lemma 5.1}
For any $u\in W^{2,2}_{0}(d\mu_{g(t)})$ , we have
\begin{align}
\int_{M}|\nabla_{g(t)}^{2} u|^{2}_{g(t)}d \mu_{g(t)}=\int_{M}\left(|\Delta_{g(t),f(t)}u|_{g(t)}^{2}-\text{\rm Ric}_{f(t)}(\nabla_{g(t)}u,\nabla_{g(t)}u)\right)d \mu_{g(t)}.\notag
\end{align}
\end{lemma}

\begin{proof}
This result has been proved in Lemma 1.13 of \cite{frequency on RF}.
\end{proof}

\begin{theorem}\label{theorem 5.2}
Let $(M,\varphi(t))_{t\in[0,T]}$ be the solution of the Laplacian $G_{2}$ flow $\eqref{G2 flow}$ on a closed $7$-dimensional manifold $M$  with $T<+\infty$ and $\displaystyle{\text{\rm Ric}_{f(t)}\leq \frac{\kappa(t)}{2h(t)}g(t)}$, where $g(t)$ is the Riemannian metric associated with $\varphi(t)$ . 

\begin{itemize}
    \item[(i)] If $h(t)$ is a negative time-dependent function, then the parabolic frequency $U(t)$ is monotone increasing along the Laplacian $G_{2}$ flow.
    \item[(ii)] If $h(t)$ is a positive time-dependent function, then the parabolic frequency $U(t)$ is monotone decreasing along the Laplacian $G_{2}$ flow.
\end{itemize}
\end{theorem}

\begin{proof}
We only give the proof of the first case (The second case is similar to it).
Our main purpose is to compute the $I'(t)$ and $D'(t)$. Under the Laplacian $G_{2}$ flow $\eqref{G2 flow}$,  combining with the linear heat equation $\eqref{lpe}$ and Lemma \ref{lemma 4.1}, we can obtain
\begin{align}
(\partial_{t}-\Delta)|\nabla u|^{2} 
&=2\langle\nabla u,\nabla(\partial_{t}-\Delta)u\rangle-2|\nabla^{2} u|^{2}+\frac{2}{3}|{\bf T}|^{2}|\nabla u|^{2}\label{3.6}\\
&\quad-4\sum_{1\leq k\leq 7}\bigg|\sum_{1\leq i\leq 7}{\bf T}_{ik}\nabla_{i}u\bigg|^{2}\notag\\
&=2a(t)|\nabla u|^{2}-2|\nabla^{2} u|^{2}+\frac{2}{3}|{\bf T}|^{2}|\nabla u|^{2}-4\sum_{1\leq k\leq 7}\bigg|\sum_{1\leq i\leq 7}{\bf T}_{ik}\nabla_{i}u\bigg|^{2}.\notag
\end{align}
From $\eqref{volume}$ and integration by parts, we get the derivative of $I(t)$ as follow
\begin{align}\label{5.3}
I'(t) &=\int_{M}\left(2u\partial_{t}u-u^{2}\frac{\Delta {\bf K}}{{\bf K}}\right)d\mu\\
&=\int_{M}\left(2u\partial_{t}u-\Delta (u^{2})\right)d\mu\notag\\
&=\int_{M}\left(2u\partial_{t}u-2u\Delta u-2|\nabla u|^{2}\right)d\mu\notag\\
&=-\frac{2}{h}D(t)+2a(t) I(t).\notag
\end{align}
If we write
$$\hat{I}(t)=\exp\left\{-\int_{t_{0}}^{t} 2a(s)ds\right\}I(t),$$
then we can easily find 
\begin{align}
\hat{I}^{\prime}(t)=-\frac{2}{h}\exp\left\{-\int_{t_{0}}^{t} 2a(s)ds\right\}D(t).
\label{derivative of I}
\end{align}

Next, it turns to compute the derivative of $D(t)$. Using $\eqref{volume}$, $\eqref{3.6}$ and the assumption of the Bakry-\'{E}mery Ricci curvature, we obtain
\begin{align}
D^{\prime}(t) &=h^{\prime}\int_{M}|\nabla u|^{2}d\mu+h\int_{M}\left(\partial_{t}|\nabla u|^{2}-|\nabla u|^{2}\frac{\Delta {\bf K}}{\textbf{K}}\right)d\mu\notag\\
&=h^{\prime}\int_{M}|\nabla u|^{2}d\mu+h\int_{M}(\partial_{t}-\Delta)|\nabla u|^{2}d\mu\notag\\
&=(2ah+h^{\prime})\int_{M}|\nabla u|^{2}d\mu-2h\int_{M}|\nabla^{2} u|^{2}d\mu+\frac{2}{3}h\int_{M}|{\bf T}|^{2}|\nabla u|^{2}d\mu\notag\\
&\quad-4h\int_{M}|{\bf T}_{ik}\nabla^{i}u|^{2}d\mu\notag\\
&\geq(2ah+h^{\prime}-\frac{2}{3}hR_{0})\int_{M}|\nabla u|^{2}d\mu-2h\int_{M}|\nabla^{2}u|^{2}d\mu\notag\\
&=(2ah+h^{\prime}-\frac{2}{3}hR_{0})\int_{M}|\nabla u|^{2}d\mu-2h\int_{M}\left[|\Delta_{f}u|^{2}-\text{\rm Ric}_{f}(\nabla u,\nabla u)\right]d\mu\notag\\
&\geq(\kappa+2ah+h^{\prime}-\frac{2}{3}hR_{0})\int_{M}|\nabla u|^{2}d\mu-2h\int_{M}|\Delta_{f}u|^{2}d\mu\notag\\
&=\left(2a-\frac{2}{3}R_{0}+\frac{h^{\prime}+\kappa}{h}\right)D(t)-2h\int_{M}|\Delta_{f}u|^{2}d\mu.\notag
\end{align}
where we write $\displaystyle{R_{0}=\min_{M\times[t_{0},t_{1}]}R(t)}$.
Similarly, if we write 
$$\widehat{D}(t)=\exp\left\{-\int_{t_{0}}^{t} \left[2a(s)-\frac{2}{3}R_{0}+\frac{h^{\prime}(s)+\kappa(s)}{h(s)}\right]ds\right\} D(t),$$
then we get
\begin{align}
\widehat{D}^{\prime}(t)\geq-2h\exp\left\{-\int_{t_{0}}^{t} \left[2a(s)-\frac{2}{3}R_{0}+\frac{h^{\prime}(s)+\kappa(s)}{h(s)}\right]ds\right\}\int_{M}|\Delta_{f}u|^{2}d\mu.
\label{derivative of D}
\end{align}

Finally, the parabolic frequency $U(t)$ can be written as $\displaystyle{U(t)=\frac{\widehat{D}(t)}{\hat{I}(t)}}$. By $\eqref{derivative of I}$ and $\eqref{derivative of D}$, we can compute the derivative of $U(t)$
\begin{align}
\hat{I}^{2}(t)U^{\prime}(t) &=\widehat{D}^{\prime}(t)\hat{I}(t)-\hat{I}^{\prime}(t)\widehat{D}(t)\\
&\geq-2h\exp\left\{-\int_{t_{0}}^{t} \left[4a(s)-\frac{2}{3}R_{0}+\frac{h^{\prime}(s)+\kappa(s)}{h(s)}\right]ds\right\}\notag\\
&\quad\cdot\left[\left(\int_{M}|\Delta_{f}u|^{2}d\mu\right)\cdot\left(\int_{M}|u|^{2}d\mu\right)-\left(\int_{M}|\nabla u|^{2}d\mu\right)^{2}\right]\notag\\
&\geq-2h\exp\left\{-\int_{t_{0}}^{t} \left[4a(s)-\frac{2}{3}R_{0}+\frac{h^{\prime}(s)+\kappa(s)}{h(s)}\right]ds\right\}\notag\\
&\quad\cdot\left[\left(\int_M\langle u(t),\Delta_{f}u\rangle d\mu\right)^{2}-\left(\int_{M}|\nabla u|^{2}d\mu\right)^{2} \right]\notag\\
&= 0.\notag
\end{align}
The last inequality is directly obtained by the definition of $D(t)$ and the Cauchy-Schwarz inequality. 
\end{proof}

Then we have the following

\begin{corollary}\label{coro 5.3}
Let $(M,\varphi(t))_{t\in[0,T]}$ be the solution of the Laplacian $G_{2}$ flow $\eqref{G2 flow}$ on a closed $7$-dimensional manifold $M$  with $T<+\infty$ and $\displaystyle{\text{\rm Ric}_{f(t)}\leq \frac{\kappa(t)}{2h(t)}g(t)}$, where $g(t)$ is the Riemannian metric associated with $\varphi(t)$ .  If $u(t_{1})=0$, then $u(t)\equiv 0$ for any $t\in[t_{0},t_{1}]\subset(0,T)$.
\end{corollary}

\begin{proof}
We give the proof of case $h(t)<0$ (The case $h(t)>0$ is similar to it).
Recalling the definition of $U(t)$, we get
\begin{align}\label{log I}
\frac{d}{dt}\ln(I(t))&=\frac{I^{\prime}(t)}{I(t)}=-\frac{2D(t)}{h(t)I(t)}+2a(t)\\
&=-\frac{2}{h(t)}\exp\left\{\int_{t_{0}}^{t}\left(-\frac{2}{3}R_{0}+\frac{h^{\prime}(s)+\kappa(s)}{h(s)} \right)ds\right\}U(t)+2a(t).\notag
\end{align}
According to Theorem \ref{theorem 5.2} and integrating $\eqref{log I}$ from $t'$ to $t_{1}$ for any $t'\in[t_{0},t_{1}]$,  yields
\begin{align}
&\quad\ln I(t_{1})-\ln I(t')\notag\\
&= -2\int_{t'}^{t_{1}}\exp\left\{\int_{t_{0}}^{t}\left(\frac{h^{\prime}(s)+\kappa(s)}{h(s)} -\frac{2}{3}R_{0}\right)ds\right\}\frac{U(t)}{h(t)}dt+2\int_{t'}^{t_{1}}a(t)dt\notag\\
&\geq -2U(t_{0})\int_{t'}^{t_{1}}\exp\left\{\int_{t_{0}}^{t}\left(\frac{h^{\prime}(s)+\kappa(s)}{h(s)} -\frac{2}{3}R_{0}\right)ds\right\}\frac{dt}{h(t)}+2\int_{t'}^{t_{1}}a(t)dt.\notag
\end{align}
Since $a(t),h(t)$  are finite, it follows from the last inequality that
\begin{align}
\frac{I(t_{1})}{I(t')}&\geq\exp \left\{-2U(t_{0})\int_{t'}^{t_{1}}\exp\left\{\int_{t_{0}}^{t}\left(\frac{h'(s)+\kappa(s)}{h(s)}-\frac{2}{3}R_{0}\right)ds\right\}\frac{dt}{h(t)}\right.\notag\\
&\quad\left.+2\int_{t'}^{t_{1}}a(t)dt\right\},\notag
\end{align}
which implies Corollary \ref{coro 5.3}.
\end{proof}

\begin{remark}
If we let
$$-\frac{2}{h(t)}\exp\left\{\int_{t_{0}}^{t}\left(\frac{h'(s)+\kappa(s)}{h(s)} -\frac{2}{3}R_{0}\right)ds\right\}\equiv C_{3},$$
and $a^{\prime}(t)\geq0$, where $C_{3}$ is a constant, then we get $\ln I(t)$ is convex, which is a parabolic version of the classical Hadamard’s three-circle theorem
for holomorphic functions. For example, if we let 
$$
h\equiv C_{4}, \ \ \ \kappa=\frac{2}{3}R_{0}C_{4},
$$
where $C_{4}$ is any constant, then we get the classical Hadamard’s three-circle theorem.
\end{remark}

\subsection{Parabolic frequency for the more general parabolic equations under Laplacian $G_{2}$ flow.}
This section considers the parabolic frequency for more general parabolic equations. We use the definition of parabolic frequency in Section 5.1, {\bf here we assume $h(t)$ is the negative smooth function.}

\begin{theorem}\label{theorem 5.5}
    Suppose that $u(t)$ satisfies 
    $$|(\partial_{t}-\Delta_{g(t)}) u(t)|\leq C(t)\left(|\nabla_{g(t)} u(t)|_{g(t)}+|u(t)|\right)$$ along the Laplacian $G_{2}$ flow $\eqref{G2 flow}$. Then
    \begin{align}
        \left(\ln I(t)\right)^{\prime}&\geq-\big(2+C(t)\big)\exp\left\{\int_{t_{0}}^{t}\left(-\frac{2}{3}R_{0}+\frac{h'(s)+\kappa(s)}{h(s)}\right) ds\right\}\frac{U(t)}{h(t)}-3C(t)\notag,\\
        U^{\prime}(t)&\geq C^{2}(t)\left(U(t)+h(t)\exp\left\{\int_{t_{0}}^{t}\left(-\frac{2}{3}R_{0}+\frac{h'(s)+\kappa(s)}{h(s)}\right) ds\right\}\right)\notag.
    \end{align}
\end{theorem}
\begin{proof}
    At first, we calculate
    \begin{align}
        I^{\prime}(t)&=\int_{M}\left(2u\partial_{t}u-2u\Delta u-2|\nabla u|^{2}\right)d\mu\notag\\
        &=-\frac{2}{h}D(t)+2\int_{M}u(\partial_{t}-\Delta)ud\mu\notag\\
        &\geq-\frac{2}{h}D(t)-2C\int_{M}|u|(|\nabla u|+|u|)d\mu\notag\\
        &=-\frac{2}{h}D(t)-2CI(t)-2C\int_{M}|\nabla u||u|d\mu\notag\\
        &\geq-\frac{2}{h}D(t)-3CI(t)-\frac{C}{h}D(t)\notag\\
        &=-\frac{2+C}{h}D(t)-3CI(t)\notag\\
        &=-(2+C)\exp\left\{\int_{t_{0}}^{t}\left(-\frac{2}{3}R_{0}+\frac{h'(s)+\kappa(s)}{h(s)}\right) ds\right\}I(t)\frac{U(t)}{h}-3CI(t)\notag.
    \end{align}
    Then we get the first inequality.

    For the second inequality, we write
    $$D(t)=-h\int_{M}u\left(\Delta_{f}+\frac{1}{2}(\partial_{t}-\Delta)\right)ud\mu{+}\frac{h}{2}\int_{M}(\partial_{t}-\Delta)ud\mu.$$
    Then from $\eqref{5.3}$, we get
    \begin{align}
        I^{\prime}(t)&=\int_{M}\left(2u\partial_{t}u-2u\Delta u-2|\nabla u|^{2}\right)d\mu\notag\\
        &=2\int_{M}u(\Delta_{f}+\partial_{t}-\Delta)ud\mu\notag\\
        &=2\int_{M}u\left(\Delta{f}+\frac{1}{2}(\partial_{t}-\Delta)\right)ud\mu+\int_{M}(\partial_{t}-\Delta)ud\mu\notag.
    \end{align}
    From the above two equalities, we get
    $$I^{\prime}(t)D(t)={{-}}2h\left(\int_{M}u\left(\Delta_{f}+\frac{1}{2}(\partial_{t}-\Delta)\right)ud\mu\right)^{2}{{+}}\frac{h}{2}\left(\int_{M}(\partial_{t}-\Delta)ud\mu\right)^{2}.$$
    According to Lemma \ref{lemma 4.1} and Lemma \ref{lemma 5.1}, we can calculate
    \begin{align}
        D^{\prime}(t)&=h^{\prime}\int_{M}|\nabla u|^{2}d\mu+h\int_{M}(\partial_{t}-\Delta)|\nabla u|^{2}d\mu\notag\\
        &=h\int_{M}\Big{(}\frac{h^{\prime}}{h}|\nabla u|^{2}-2|\nabla^{2}u|^{2}+2\langle\nabla u,\nabla(\partial_{t}-\Delta)u\rangle\notag\\
        &\quad+\frac{2}{3}|{\bf T}|^{2}|\nabla u|^{2}-4|{\bf T}_{ik}\nabla_{i}u|^{2}\Big{)}d\mu\notag\\
        &=h\int_{M}\Big{(}\frac{h^{\prime}}{h}|\nabla u|^{2}-2|\Delta_{f}u|^{2}+2{\rm Ric}_{f}(\nabla u,\nabla u)-2\Delta_{f}u\cdot(\partial_{t}-\Delta)u\notag\\
        &\quad+\frac{2}{3}|{\bf T}|^{2}|\nabla u|^{2}-4|{\bf T}_{ik}\nabla_{i}u|^{2}\Big{)}d\mu\notag\\
        &=2h\int_{M}\left({\rm Ric}_{f}(\nabla u,\nabla u)+\frac{h^{\prime}}{2h}|\nabla u|^{2}+\frac{{1}}{3}|{\bf T}|^{2}|\nabla u|^{2}-{{2}}|{\bf T}_{ik}\nabla_{i}u|^{2}\right)d\mu\notag\\
        &\quad-2h\int_{M}\left(|\Delta_{f}u|^{2}+\Delta_{f}u\cdot(\partial_{t}-\Delta)u\right)d\mu\notag\\
        &\geq \int_{M}\left(\kappa+h^{\prime}-\frac{2}{3}{{h}}R_{0}\right)|\nabla u|^{2}d\mu-2h\int_{M}\left(|\Delta_{f}u|^{2}+\Delta_{f}u\cdot(\partial_{t}-\Delta)u\right)d\mu\notag\\
        &=-2h\int_{M}\left(\left|\left(\Delta_{f}+\frac{1}{2}(\partial_{t}-\Delta)\right)u\right|^{2}-\frac{1}{4}|(\partial_{t}-\Delta)u|^{2}\right)d\mu\notag\\
        &\quad+\left(\kappa+h^{\prime}-\frac{2}{3}{{h}}R_{0}\right)\frac{D(t)}{h}\notag.      
    \end{align}
where $\displaystyle{R_{0}=\min_{M\times[t_{0},t_{1}]}R(t)}$.

Together with the above, using Cauchy-Schwarz inequality, yields
\begin{align}
    I^{2}(t)U^{\prime}(t)&=\exp\left\{-\int_{t_{0}}^{t}\left(-\frac{2}{3}R_{0}+\frac{h'(s)+\kappa(s)}{h(s)}\right) ds\right\}\notag\\
    &\quad\cdot\left[\left(\frac{2}{3}R_{0}-\frac{h'(s)+\kappa(s)}{h(s)}\right)I(t)D(t)+I(t)D^{\prime}(t)-I^{\prime}(t)D(t)\right]\notag\\
    &\geq \exp\left\{-\int_{t_{0}}^{t}\left(-\frac{2}{3}R_{0}+\frac{h'(s)+\kappa(s)}{h(s)}\right) ds\right\}\notag\\
    &\quad\cdot\left\{-2hI(t)\int_{M}\left(\left|\left(\Delta_{f}+\frac{1}{2}(\partial_{t}-\Delta)\right)u\right|^{2}-\frac{1}{4}|(\partial_{t}-\Delta)u|^{2}\right)d\mu\right.\notag\\
    &\quad{{+}}\left.2h\left(\int_{M}u\left(\Delta_{f}+\frac{1}{2}(\partial_{t}-\Delta)\right)ud\mu\right)^{2}-\frac{h}{2}\left(\int_{M}(\partial_{t}-\Delta)ud\mu\right)^{2}\right\}\notag\\
    &\geq\frac{h}{2}I(t)\exp\left\{-\int_{t_{0}}^{t}\left(-\frac{2}{3}R_{0}+\frac{h'(s)+\kappa(s)}{h(s)}\right) ds\right\}\int_{M}|(\partial_{t}-\Delta)u|^{2}d\mu\notag\\
    &\geq\frac{h}{2}I(t)C^{2}\exp\left\{-\int_{t_{0}}^{t}\left(-\frac{2}{3}R_{0}+\frac{h'(s)+\kappa(s)}{h(s)}\right) ds\right\}\int_{M}(|\nabla u|+|u|)^{2}d\mu\notag\\
    &\geq C^{2}\exp\left\{-\int_{t_{0}}^{t}\left(-\frac{2}{3}R_{0}+\frac{h'(s)+\kappa(s)}{h(s)}\right) ds\right\}I(t)\big{(}D(t)+hI(t)\big{)}\notag.
\end{align}
Then we prove this Theorem. 
\end{proof}

\begin{corollary}
    Suppose that $u(t):M\times[t_{0},t_{1}]\rightarrow \mathbb{R}$ satisfies 
    $$|(\partial_{t}-\Delta_{g(t)}) u(t)|\leq C(t)\left(|\nabla_{g(t)} u(t)|_{g(t)}+|u(t)|\right)$$ along the Laplacian $G_{2}$ flow $\eqref{G2 flow}$. Then
    \begin{align}
        I(t_{1})&\geq I(t_{0})\exp\left\{\int_{t_{0}}^{t_{1}}-\big(2+\sup_{[t_{0},t_{1}]}C(t)\big)\exp\left\{\int_{t_{0}}^{t}\left(-\frac{2}{3}R_{0}+\frac{h'(s)+\kappa(s)}{h(s)}\right) ds\right\}\right.\notag\\  
        &\quad\cdot\frac{1}{h{{(t)}}}\left[{{\exp\left\{-\int_{t_{0}}^{t}C^{2}(s)ds\right\}}}U(t_{0})+{{\exp\left\{-\int_{t_{0}}^{t}C^{2}(s)ds\right\}}}\right.\notag\\
        &\quad\cdot\left.\int_{t_{0}}^{t_{1}}\exp\left\{\int_{t_{0}}^{\tau}\left(-\frac{2}{3}R_{0}+\frac{h'(s)+\kappa(s)}{h(s)}{{-C^{2}(s)}}\right) ds \right\}d\tau\right]dt \notag\\
        &\quad\left.-3\int_{t_{0}}^{t_{1}}\sup_{[t_{0},t_{1}]}C(t)dt\right\}\notag.
    \end{align}
    In particular, if $u(t_{1})=0$, then $u\equiv 0$ for all $t\in[t_{0},t_{1}]$.
\end{corollary}
\begin{proof}
    By the first inequality in Theorem \ref{theorem 5.5}, we get
    \begin{align}\label{5.8}
        \ln(I(t_{1}))&-\ln(I(t_{0}))\\
        &\geq\int_{t_{0}}^{t_{1}}-(2+C{{(t)}})\exp\left\{\int_{t_{0}}^{t}\left(-\frac{2}{3}R_{0}+\frac{h'(s)+\kappa(s)}{h(s)}\right) ds\right\}\frac{U(t)}{h(t)}dt\notag\\
        &\quad-\int_{t_{0}}^{t_{1}}3C{{(t)}}dt\notag.
    \end{align}
If we write $\widehat{U}(t)$ as
$\widehat{U}(t)=\exp\{-\int_{t_{0}}^{t}C^{2}(s)ds\}U(t)$, then from the second inequality in Theorem \ref{theorem 5.5}, we have 
\begin{align}\label{5.9}
\widehat{U}(t)^{\prime}\geq h{{(t)}}C^{2}{{(t)}}\exp\left\{\int_{t_{0}}^{t}\left(-\frac{2}{3}R_{0}+\frac{h'(s)+\kappa(s)}{h(s)}{{-C^{2}(s)}}\right) ds\right\}.
\end{align}
Integrating $\eqref{5.9}$ from $t_{0}$ to $t$ for any $t\in[t_{0},t_{1}]$,  yields
\begin{align}
    \widehat{U}(t)\geq \widehat{U}(t_{0})+\int_{t_{0}}^{t}h{{(\tau)}}C^{2}{{(\tau)}}\exp\left\{\int_{t_{0}}^{\tau}\left(-\frac{2}{3}R_{0}+\frac{h'(s)+\kappa(s)}{h(s)}{{-C^{2}(s)}}\right) ds\right\}d\tau,\notag
\end{align}
which means
\begin{align}\label{5.10}
    U(t)&\geq {{\exp\left\{-\int_{t_{0}}^{t}C^{2}(s)ds\right\}}}U(t_{0})+{{\exp\left\{-\int_{t_{0}}^{t}C^{2}(s)ds\right\}}}\\
    &\quad\cdot\int_{t_{0}}^{t_{1}}h{{(\tau)}}C^{2}{{(\tau)}}\exp\left\{\int_{t_{0}}^{\tau}\left(-\frac{2}{3}R_{0}+\frac{h'(s)+\kappa(s)}{h(s)}{{-C^{2}(s)}}\right) ds\right\}d\tau\notag.
\end{align}
where we use $h(t)$ is the negative smooth function.
Submitting $\eqref{5.10}$ to $\eqref{5.8}$, we get the desired result.
\end{proof}

\textbf{Acknowledgments.}\ \ 
The first author is supported by the National Key R$\&$D Program of China 2020YFA0712800, the Interdisciplinary Research Foundation for Doctoral Candidates of Beijing Normal University (Grant BNUXKJC2318) and the National Natural Science Foundation of China (NSFC12371049). The second author is funded by the Shanghai Institute for Mathematics and Interdisciplinary Sciences(SIMIS) under grant number SIMIS-ID-2024-LG. The authors would like to thank the referee for the valuable comments and suggestions.

\textbf{Statements and Declarations.}\ \
No conflict of interest.

\textbf{Data availability statement.}\ \
Data availability is not applicable to this manuscript as no new data was created or analyzed in this study.

\bibliographystyle{amsplain}

\end{document}